\documentclass{compositio}

\usepackage{fancyhdr, amsmath, amssymb, mathtools, yfonts}
\usepackage{mathrsfs}
\usepackage{graphicx}
\usepackage{tikz}
\usepackage[all]{xy}
\usepackage[utf8]{inputenc}
\usepackage{amsthm}
\usepackage[english]{babel}
\usepackage{chngcntr}
\usepackage{ifthen}
\usepackage{calc}
\usepackage{hyperref}
\usepackage{tikz-cd}
\numberwithin{equation}{section}

\newcommand{\Z}{\mathbb{Z}}
\newcommand{\Q}{\mathbb{Q}}

\newcommand\FF{\mathbb{F}}

\newcommand\pp{\mathfrak{p}}

\newcommand\Aut{\text{Aut}}

\newcommand\Gal{\mathrm{Gal}}

\newcommand\ord{\mathrm{ord}}

\usepackage[OT2,T1]{fontenc}
\DeclareSymbolFont{cyrletters}{OT2}{wncyr}{m}{n}
\DeclareMathSymbol{\Sha}{\mathalpha}{cyrletters}{"58}

\newtheorem{lemma}{Lemma}[section]
\newtheorem{theorem}[lemma]{Theorem}
\newtheorem{proposition}[lemma]{Proposition}
\newtheorem{corollary}[lemma]{Corollary}
\newtheorem{mydef}[lemma]{Definition}

\newtheorem*{remark}{Remark}

\newcommand{\sump}{\sideset{}{^\prime}\sum}

\begin{document}
%
%

\title{Weak approximation on the norm one torus}
\author{P. Koymans}
\email{peter.koymans@eth-its.ethz.ch}
\address{Institute for Theoretical Studies, ETH Zurich, 8006 Zurich, Switzerland}
\author{N. Rome}
\email{rome@tugraz.at}
\address{Graz University of Technology, Institute of Analysis and Number Theory,
Kopernikusgasse 24/II, 8010 Graz, Austria.}
\dedication{Dedicated to Ad\`ele Koymans-Funcken}

\begin{abstract}
For any abelian group $A$, we prove an asymptotic formula for the number of $A$-extensions $K/\Q$ of bounded discriminant such that the associated norm one torus $R_{K/\Q}^1 \mathbb{G}_m$ satisfies weak approximation. We are also able to produce new results on the Hasse norm principle and to provide new explicit values for the leading constant in some instances of Malle's conjecture.
\end{abstract}

\maketitle

\section{Introduction}
Fix a finite abelian group $A$ of order $d$. Let $K/\Q$ be a Galois extension with $\Gal(K/\Q) \cong A$. We can view $K$ as $d$-dimensional $\Q$ vector space. Suppose that $\omega_1, \ldots, \omega_d$ is a basis then consider the affine variety $T$ over $\Q$ defined by the equation
$
N_{K/\Q}(x_1\omega_1 + \ldots + x_n \omega_n) =1.
$
(Alternatively, $T$ is the kernel $R^1_{K/\Q} \mathbb{G}_m$ of the norm map from the Weil restriction of scalars $R_{K/\Q} \mathbb{G}_m \rightarrow \mathbb{G}_m$.)
The arithmetic of this torus is particularly interesting as both a question in the study of rational points on Fano varieties and in the context of arithmetic statistics. The close parallels between counting number fields of bounded discriminant and studying rational points on varieties has been the subject of much recent interest (see \cite{Ellenberg} and \cite{Yasuda}). The torus will always have rational points so we can ask: how are the rational points distributed (qualitatively) on $T$?

\begin{mydef}
Let $X/k$ be a smooth variety. We say that \emph{weak approximation} holds for $X$ if the rational points $X(k)$ are dense in the product of local points $\prod_{\nu} X(k_{\nu})$, under the product topology.
\end{mydef}

If $A$ is cyclic, then weak approximation is guaranteed on $T$ by the Hasse norm theorem in class field theory (or by \eqref{eq:vosk}). However, it was recently shown by Frei--Loughran--Newton~\cite[Thm 1.5]{FLN} that for any non-cyclic abelian group $A$, there exist extensions $K/\Q$ with Galois group $A$ such that the associated norm one torus fails to satisfy weak approximation. Our main result is to establish an asymptotic formula for precisely how many $A$-extensions, when ordered by absolute discriminant, are such that weak approximation holds on the norm one torus.

\begin{theorem}
\label{thm:main}
Let $A$ be a non-trivial finite abelian group and $\ell$ the smallest prime divisor of $\vert A \vert$. There exist constants $C(A)$, $\delta(A)$ and $\alpha(A)$ all positive such that for all $X \geq 100$, we have
\begin{align*}
\#&\{ K/\Q : \textup{Disc}(K/\Q)\leq X, \textup{Gal}(K/\Q) \cong A \textup{ and } R^1_{K/\Q} \mathbb{G}_m \textup{ satisfies weak approximation}\}\\
&
=
C(A) X^{\frac{\ell}{|A|\cdot (\ell -1)}} (\log X)^{\alpha(A)-1} + O(X^{\frac{\ell}{|A|\cdot (\ell -1)}}(\log X)^{\alpha(A) -1- \delta(A)}).
\end{align*}
\end{theorem}

\noindent The constant $\alpha(A)$ has the following explicit expression
\begin{align}
\label{ealphaA}
\alpha(A) = \sum_{a \in A[\ell] - \{0\}} \frac{|\{b \in A : \textup{Hom}(\wedge^2(A), \Q/\Z) \rightarrow \textup{Hom}(\wedge^2(\langle a, b \rangle), \Q/\Z) \textup{ is the zero map}\}|}{(\ell-1)\cdot|A|},
\end{align}
where $\wedge^2$ is the second exterior algebra and $\textup{Hom}(\wedge^2(A), \Q/\Z) \rightarrow \textup{Hom}(\wedge^2(\langle a, b \rangle), \Q/\Z)$ is simply the natural restriction map.

Frei--Loughran--Newton~\cite[Thm 1.5]{FLN} established that $0\%$ of $A$-extensions (of any given number field) satisfy the weak approximation property when ordered by discriminant if the $\ell$-Sylow subgroup of $A$ is not cyclic (and otherwise a positive proportion do). This result is recovered for abelian extensions of $\Q$ by combining our result with Wright's theorem on the number of $A$-extensions of bounded discriminant \cite[Thm 1.2]{Wright}. Their results gave density statements but no information about the order of magnitude of the size of the set in Theorem \ref{thm:main}. The only case where an asymptotic formula had previously been known is when $A \cong \left(\Z/2\Z\right)^2$, by work of the second named author \cite{biquad}.

\subsection{The Hasse norm principle}\label{ss:hnp}
The problem of weak approximation on $R^1_{K/\Q} \mathbb{G}_m$ is closely related to determining whether the \emph{Hasse norm principle} holds for $K/\Q$. This problem asks: if  an element of $K$ is a norm for every local extension $K_{\mathfrak{p}}/\Q_p$, is it also a norm for the global extension $K/\Q$? In geometric language, do all the principal homogeneous spaces for the torus $R^1_{K/\Q} \mathbb{G}_m$ satisfy the Hasse principle? The connection between the two problems is made explicit by the following short exact sequence due to Voskresenskii~\cite[Thm 6]{Vosk} 
\begin{equation}
\label{eq:vosk}
0 \rightarrow A(T) \rightarrow H^3(A, \Z)^\vee \rightarrow \Sha(T) \rightarrow 0,
\end{equation} 
where $\Sha(T)$ is the Tate--Shafarevich group of the torus and $A(T)= \left( \prod_{\nu} T(\Q_{\nu})\right)/\overline{T(\Q)}$ is the defect of weak approximation. This sequence can also be viewed as an artifact of the fact that the Brauer--Manin obstruction is the only obstruction to weak approximation for the torus and to the Hasse principle for any principal homogeneous space under the torus (see \cite{Sansuc}).

If $K/\Q$ is cyclic then Tate's theorem shows that $H^3(A, \Z)^\vee$ is trivial, guaranteeing both the Hasse norm principle and weak approximation on the norm one torus. If the group $H^3(A, \Z)^\vee$ is cyclic of prime order then either $\Sha(T)=0$ or $A(T)=0$ but not both. In other words, the Hasse norm principle fails if and only if $R^1_{K/\Q} \mathbb{G}_m$ satisfies weak approximation. Therefore, in certain cases, one can deduce information about Hasse norm principle failure from our main theorem.

\begin{corollary}
Suppose that $A \cong \prod_{i=1}^r (C_{p_i})^{n_i}$ for primes $p_1<\ldots<p_r$ and natural numbers $n_i$ such that $\prod_{i=1}^r n_i = 2$. Then there exist constants $C(A)$, $\delta(A)$ and $\alpha(A)$ all positive such that for all $X \geq 100$, we have
\[
\#\{ K/\Q : \textup{Disc}(K/\Q)\leq X, \textup{Gal}(K/\Q) \cong A \textup{ and } K/\Q \textup{ fails the Hasse norm principle}\}
\]
\[
=
C(A) X^{\frac{\ell}{|A| \cdot (\ell -1)}} (\log X)^{\alpha(A)-1} + O(X^{\frac{\ell}{|A| \cdot (\ell -1)}}(\log X)^{\alpha(A) -1 - \delta(A)}).
\]
Moreover, if $n_1 \neq 2$, then a positive proportion of $A$-extensions fail the Hasse norm principle.
\end{corollary}

When combined with Wright's theorem \cite{Wright}, this corollary recovers density results due to Frei--Loughran--Newton~\cite[Thm 1.1]{FLN} and the authors~\cite[Thm 1.2]{KR} for such $A$. In $\S$ \ref{sec:explicit}, we give an explicit expression for the proportion of $C_2 \times C_3 \times C_3$-extensions failing the Hasse norm principle.

\subsection{Proof structure}
The arithmetic of the norm one torus is intimately connected with the arithmetic of the field $K$. Write $D_p$ for the decomposition group of a prime $p$.

\begin{lemma}
The norm one torus $R^1_{K/\Q}$ satisfies weak approximation if and only if the restriction map $\textup{Hom}(\wedge^2(A), \Q/\Z) \rightarrow \textup{Hom}(\wedge^2(D_p), \Q/\Z)$ is identically zero for all primes $p$.
\end{lemma}

This lemma provides the key criterion by which we may distinguish fields with the weak approximation property. Therefore our main result can be re-interpreted as a problem concerning counting the number of $A$-extensions, ordered by discriminant, with a set of allowed splitting conditions at every finite place. In general, such problems are extremely difficult and of great interest to the arithmetic statistics community. We are able to achieve success in this particular case by a combination of a clever parametrisation of the fields and detecting our splitting conditions using character sums.

More specifically, in Section \ref{sec:par}, $A$-extensions will be parametrised by tuples of squarefree integers $(v_a)_{a \in A - \{0\}}$. In the formula for the discriminant, the higher the order of the group element $a$, the larger the exponent of the component $v_a$. Therefore the components $v_a$ with $a \in A[\ell] - \{0\}$ will carry the most weight in the formula and we assume that the remaining variables are quite small. In other words, we may fix an extension $L/\Q$ with $\Gal(L/\Q) \cong A/A[\ell]$ of small discriminant and allow $K$ to vary across $A[\ell]$-extensions of $L$. This reduces the problem to understanding the number of multicyclic extensions of bounded discriminant with certain splitting conditions at each prime. This also explains the shape of the logarithmic exponent given in equation (\ref{ealphaA}), since inertia subgroups will typically live in $A[\ell]$.

The nature of the splitting conditions imposed on the multicyclic extension is encapsulated in the notion of \emph{$f$-correctness} in Section \ref{sec:reduction}. Theorem \ref{tMultiReduction} establishes the necessary count for multicyclic extensions and this is used to establish the general statement in Theorem \ref{tMain}. Sections \ref{sec:charsum}--\ref{sec:oscillation} are devoted to proving Theorem \ref{tMultiReduction}. Finally, in Section \ref{sec:explicit}, we compute some explicit examples, elucidating the proof strategy and providing completely explicit leading constants in certain cases.

\begin{remark}
The fact that we can essentially reduce the problem to the multicyclic case is a consequence of the discriminant being an unfair counting function (in the language of Wood~\cite{Wood}). It is very plausible that a similar approach will be fruitful when ordering fields in another way, however the task is made simpler under this parametrisation when ordering by discriminant. 
\end{remark}

\begin{remark}
Frei--Loughran--Newton \cite{FLN2} also considered such problems when ordering fields by conductor instead of discriminant. In this setting, $0\%$ of fields have weak approximation on the norm one torus, for any non-cyclic abelian group $A$. If $A$ is a multicyclic group, then the discriminant is a fixed power of the conductor. As such our main results (in particular, the calculations in $\S$ \ref{sec:multicyc}) recover (and improve) this density result for extensions over $\Q$. 
\end{remark}


It is also interesting to consider the problem of the statistics of the Hasse norm principle and weak approximation on the norm one torus for non-abelian extensions. Macedo showed, in his thesis \cite{Andre}, that $100\%$ of $D_4$ octics over $\Q$ satisfy the Hasse norm principle and that $0\%$ satisfy the weak approximation property, when ordering by either conductor or discriminant. The conductor ordering part of his work is unconditional (thanks to work of Altu{\u g}--Shankar--Varma--Wilson~\cite{D4s}) but the discriminant ordering part is conditional on work in progress of Shankar--Varma~\cite{ShankarVarma} on Malle's conjecture for $D_4$ octics. Macedo--Newton \cite{AnSn} have given criteria for the failure of weak approximation and the Hasse norm principle in fields with normal closure $S_n$ and $A_n$. Combining this with the counting techniques of Bhargava~\cite{Bhargava4, Bhargava5}, Newton--Varma (in forthcoming work \cite{Varma}) will study the frequency of Hasse norm principle failures in non-quartic $S_4$ fields and non-quintic $S_5$ fields. (Note that the Hasse norm principle is guaranteed in degree $n$ $S_n$ extensions, see Voskresenski{\u \i}~\cite{Sn}, and in degree $n$ $D_n$ extensions by Bartels~\cite{Bartels}.)  There has also been recent work of Monnet~\cite{S4s} on the related problem of counting how often certain prescribed elements of a number field $k$ are norms as one varies over $S_4$ quartic extensions of $k$.

\subsection{Equidistribution of Frobenius}
As mentioned above, the key step in our proof is a reduction to the case of multicyclic extensions. It will be important that we can count multicyclic extensions which have the necessary local properties to ensure that weak approximation is satisfied for the full extension. This is the crux of Theorem \ref{tMultiReduction}, the main technical input into the proof, which should be viewed as a \emph{quantitative Frobenius equidistribution} result. Essentially it states that one can uniformly count multicyclic extensions of a number field $F$ while imposing that the Frobenius element of primes ramifying in the multicyclic extension or the number field $F$ lands in (essentially) any given subset of the Galois group. This equidsitribution is reflected nicely in the leading constants for these problems.

The total number of $\ell$-multicyclic fields (for $\ell \neq 2$) of bounded discriminant has the following leading constant in its asymptotic formula
\[
\frac{ \left( 1 + \frac{\ell^n-1}{\ell^2} \right)(\ell^n - \ell^{n-1})^{-\frac{\ell^n -1}{\ell -1} +1}}{\Gamma \left(\frac{\ell^n-1}{\ell-1} \right)\prod_{i=0}^{n-1} (\ell^n - \ell^i)}
\prod_{p \equiv 1 \bmod \ell} \left( 1 + \frac{\ell^n -1}{p} \right) \prod_p \left( 1 - \frac{1}{p} \right)^{\frac{\ell^n-1}{\ell-1}}
.
\]
This fact is proven in Theorem \ref{thm:multicycWright}. When counting such extensions for which weak approximation holds on the norm one torus the leading constant (c.f. $\S$ \ref{sec:multicyc}) is
\[
\frac{\left( 1 + \frac{\ell^n-1}{\ell^2} \right)(\ell^n - \ell^{n-1})^{-\frac{\ell^n-1}{\ell^{n-1}(\ell-1)}+1}}{\Gamma\left(\frac{\ell^n-1}{\ell^{n-1}(\ell-1)} \right)\prod_{i=0}^{n-1} (\ell^n - \ell^i)} 
\prod_{p \equiv 1 \bmod \ell} \left( 1 + \frac{\ell^n-1}{\ell^{n-1} p} \right) \prod_p \left( 1- \frac{1}{p} \right)^{\frac{\ell^n-1}{\ell^{n-1}(\ell-1)}}\!\!.
\]
One notes that the constants are remarkably similar and that one major change is to the terms within the Euler product. The factors in the Euler product corresponding to the weak approximation count feature an extra $\frac{1}{\ell^{n-1}}$ which reflects the fact that at each prime $p$, if Frobenius elements were distributed uniformly at random among all elements of the Galois group (quotiented by inertia), the probability that the Frobenius element is trivial is $\frac{1}{\ell^{n-1}}$. More generally, Theorem \ref{tMultiReduction} features a product of factors corresponding to the probability that Frobenius lands in the specified subgroups at each prime. This quantitative equidistribution is key for our purposes but also likely to be highly useful in many further problems in arithmetic statistics.

\subsection{Notations and conventions}
\begin{itemize}
\item The symbol $v$ denotes a place of $\Q$, and $p$ denotes a finite place of $\Q$. Given a finite place $p$, we write $v_p$ for the corresponding valuation.
\item We say that an integer $d$ is squarefree if $p \mid d$ implies $p^2 \nmid d$. In particular squarefree integers may be negative.
\item We say that two squarefree integers $d, e$ are coprime if we have $\gcd(d, e) = 1$ and furthermore $d$ and $e$ are not both negative.
\item We write $\infty$ for the infinite place of $\Q$. We say that $\infty$ divides $d$, written $\infty \mid d$, if $d < 0$. Then we have that two squarefree integers $d, e$ are coprime if and only if there does not exist a place $v$ of $\Q$ such that $v \mid d$ and $v \mid e$.
\item A Galois extension $K/\Q$ is called multicyclic if $\Gal(K/\Q) \cong (\Z/\ell\Z)^n$ for some prime $\ell$.
\item Elements $\phi$ of $\text{Epi}(G_\Q, A)$ will often be referred to as $A$-extensions. This is the same data as a Galois extension $K/\Q$ together with an isomorphism between $\Gal(K/\Q)$ and $A$. In case a property depends only on the field $K$, then we shall frequently abuse notation by also referring to $K/\Q$ as an $A$-extension by forgetting the choice of isomorphism.
\end{itemize}

\subsection*{Acknowledgements} 
PK would like to thank Efthymios Sofos for valuable discussions on \cite{GK, Kou}. NR would like to thank Andre Macedo for a number of discussions at the inception of this project. Both authors would like to thank Christopher Frei, Daniel Loughran and Rachel Newton for their comments. PK acknowledges the support of Dr. Max R\"ossler, the Walter Haefner Foundation and the ETH Z\"urich Foundation. NR is funded by FWF project ESP 441-NBL. This work was completed while the authors were based at the University of Michigan.

\section{Parametrisation of abelian extensions}
\label{sec:par}
The following parametrisation is based on the methods of Koymans--Pagano \cite{KPMalle}. Fix an algebraic closure $\overline{\Q}$ of $\Q$ and fix a finite abelian group $A$, which we view as a topological group by using the discrete topology. Our goal is to describe the set
\[
\{K : \Gal(K/\Q) \cong A, D_K \leq X\},
\]
where all our number fields $K$ are implicitly taken inside $\overline{\Q}$. There is a natural surjective map $\psi$
\[
\text{Epi}(G_\Q, A) \rightarrow \{K : \Gal(K/\Q) \cong A\}
\]
from the set of continuous epimorphisms $G_\Q \rightarrow A$ to $\{K : \Gal(K/\Q) \cong A\}$. The map $\psi$ sends a continuous epimorphism $\phi$ to the fixed field of $\text{ker}(\phi)$. If we define the discriminant of $\phi: G_\Q \rightarrow A$ to be the discriminant of the fixed field, then this map trivially preserves the discriminant. Furthermore, a field $K$ with $\Gal(K/\Q) \cong A$ has precisely $|\text{Aut}(A)|$ pre-images under $\psi$. Hence we will now shift our attention to $\text{Epi}(G_\Q, A)$.

It turns out to be slightly easier to work with $\text{Hom}(G_\Q, A)$, and we will later deduce results for $\text{Epi}(G_\Q, A)$ from this. We will now create a bijection between $\text{Hom}(G_\Q, A)$ and certain tuples of integers. Let us first define this space $\mathcal{A}$.

\begin{mydef}
Let $A$ be a finite abelian group. Let $\mathcal{A}$ be the set of tuples $(v_a)_{a \in A - \{0\}}$ satisfying the following conditions
\begin{itemize}
\item $v_a$ is a squarefree integer for every $a \in A - \{0\}$;
\item $v_a$ and $v_b$ are coprime for all $a, b \in A - \{0\}$ with $a \neq b$;
\item we have
\[
p \equiv 1 \bmod \frac{\ord(a)}{p^{v_p(\ord(a))}}
\]
for all prime divisors $p$ of $v_a$;
\item if $\textup{ord}(a) > 2$, then $v_a > 0$.
\end{itemize}
\end{mydef}

To create the bijection, we make some arbitrary choices. Let 
\[
K_p := 
\begin{cases}
\Q(\zeta_{p^\infty}) & \text{if } p > 2 \\
\Q(\zeta_{2^\infty})^+ & \text{if } p = 2,
\end{cases}
\]
where $\Q(\zeta_{2^\infty})^+$ denotes the maximal real subfield of $\Q(\zeta_{2^\infty})$. We have an isomorphism
\[
\Gal(K_p/\Q) \cong \Z_p \oplus \Z/(p - 1)\Z.
\]
Since $\Gal(K_p/\Q)$ is pro-cyclic, we may choose a topological generator $\tau_p$ of $\Gal(K_p/\Q)$. Let $\tau_\infty$ be a generator of $\Gal(\Q(\sqrt{-1})/\Q)$. Then the maximal abelian extension $\Q^{\text{ab}}$ of $\Q$ is the compositum of $K_p$ and $\Q(\sqrt{-1})$ over all primes $p$. Furthermore, these fields are linearly disjoint, which gives an isomorphism
\[
\Gal(\Q^{\text{ab}}/\Q) \cong \Gal(\Q(\sqrt{-1})/\Q) \times \prod_p \Gal(K_p/\Q).
\]
Define $\sigma_p$ and $\sigma_\infty$ to be the unique elements in $\Gal(\Q^{\text{ab}}/\Q)$ that project to respectively $\tau_p$, $\tau_\infty$ and zero everywhere else. Then the $\sigma_p$ and $\sigma_\infty$ together form a minimal set of topological generators of $\Gal(\Q^{\text{ab}}/\Q)$.

Let $p$ be a prime, let $n \geq 0$ and let $\ell$ be a prime congruent to $1 \bmod p^n$. Let
\[
\psi_{\ell, p, n} \in \text{Hom}(G_\Q, \Z/p^n\Z)
\]
be the unique homomorphism that is unramified away from $\ell$ and satisfies $\psi_{\ell, p, n}(\sigma_\ell) = 1$. Note that it makes sense to evaluate $\psi_{\ell, p, n}$ in $\sigma_\ell$, since any homomorphism $\psi_{\ell, p, n}: G_\Q \rightarrow \Z/p^n\Z$ must factor through $\Gal(\Q^{\text{ab}}/\Q)$. Similarly, let
\[
\psi_{p, p, n} \in \text{Hom}(G_\Q, \Z/p^n\Z)
\]
be the unique homomorphism that is unramified away from $p$ and sends $\sigma_p$ to $1$. Finally, let
\[
\psi_{\infty, 2, 1} \in \text{Hom}(G_\Q, \Z/2\Z)
\]
be the unique surjective homomorphism that factors through $\Gal(\Q(\sqrt{-1})/\Q)$.

If $x$ is a squarefree integer such that all its prime divisors are $1, p \bmod p^n$ and such that $x > 0$ if $(p, n) \neq (2, 1)$, we define
\[
\psi_{x, p, n} = \sum_{\ell \mid x} \psi_{\ell, p, n} \in \text{Hom}(G_\Q, \Z/p^n\Z).
\]
In case $\ell \neq p$ is not congruent to $1 \bmod p^n$, we let $m \geq 0$ be the largest integer such that $\ell$ is congruent to $1 \bmod p^m$. Then we define $\psi_{\ell, p, n}$ to be any lift of $\psi_{\ell, p, m}$, which means that $q _{n, m} \circ \psi_{\ell, p, n} = \psi_{\ell, p, m}$, where $q_{n, m}$ is the unique map $\Z/p^n\Z \rightarrow \Z/p^m\Z$ sending $1$ to $1$. We let $\psi_{\infty, 2, n}$ be any lift of $\psi_{\infty, 2, 1}$, while $\psi_{\infty, p, n}$ is defined to be any lift of the zero map for $p \neq 2$. We may then still define $\psi_{x, p, n}$ as above.

Using these choices, we will construct a map $\text{Par}: \mathcal{A} \rightarrow \text{Hom}(G_\Q, A)$. Take $(v_a)_{a \in A - \{0\}} \in \mathcal{A}$. Choose a cyclic quotient $\Z/p^n\Z$ of $A$. Write $\pi: A \rightarrow \Z/p^n\Z$ for the quotient map and $\pi^\ast$ for the induced map $\text{Hom}(G_\Q, A) \rightarrow \text{Hom}(G_\Q, \Z/p^n\Z)$. Then we demand that
\begin{align}
\label{eFundPar}
\pi^\ast(\text{Par}((v_a)_{a \in A - \{0\}})) = \sum_{a \in A - \{0\}} \pi(a) \cdot \psi_{v_a, p, n}
\end{align}
for all $\pi$. We emphasize that each $\pi(a) \cdot \psi_{v_a, p, n}$ is a homomorphism despite the fact that $\psi_{v_a, p, n}$ need not be. 

We claim that equation (\ref{eFundPar}) uniquely specifies the homomorphism $\text{Par}((v_a)_{a \in A - \{0\}})$. By the fundamental theorem of abelian groups, we may decompose $A$ as
\[
A = \bigoplus_{i = 1}^k \Z/p_i^{e_i}\Z.
\]
Writing $\pi_i$ for the natural projection map $A \rightarrow \Z/p_i^{e_i}\Z$, we see that each homomorphism $\pi_i \circ \text{Par}((v_a)_{a \in A - \{0\}})$ is determined by equation (\ref{eFundPar}). Therefore equation (\ref{eFundPar}) specifies at most one homomorphism $\text{Par}((v_a)_{a \in A - \{0\}})$.

To prove existence, we observe that there certainly exists a homomorphism $\text{Par}((v_a)_{a \in A - \{0\}})$ satisfying equation (\ref{eFundPar}) for $\pi \in \{\pi_1, \dots, \pi_k\}$. Viewing each $\pi_i$ as taking values in $\Q/\Z$ by fixing an inclusion $\Z/p_i^{e_i}\Z \rightarrow \Q/\Z$, we see that each map $\pi: A \rightarrow \Z/p^n\Z \rightarrow \Q/\Z$ is a $\Z$-linear combination of the $\pi_i$, since the maps $\pi_i$ generate the dual space $\text{Hom}(A, \Q/\Z)$. Therefore equation (\ref{eFundPar}) holds for all $\pi$ by linearity. 

We conclude that $\text{Par}$ is well-defined. Furthermore, we have the key property 
\begin{align}
\label{eInertiaPar}
\left(\sum_{a \in A - \{0\}} \pi(a) \cdot \psi_{v_a, p, n}\right)(\sigma_\ell) = 
\begin{cases}
\pi(a) & \text{if } a \in A - \{0\} \text{ satisfies } \ell \mid v_a \\
0 & \text{otherwise.}
\end{cases}
\end{align}

\begin{theorem}
The map $\textup{Par}$ is a bijection.
\end{theorem}

\begin{proof}
We will construct an explicit inverse $\text{Ev}$ of $\text{Par}$. Take some $\phi \in \text{Hom}(G_\Q, A)$. Then $\text{Ev}(\phi)$ is the unique tuple $(v_a)_{a \in A - \{0\}}$ of squarefree integers satisfying the property
\[
p \mid v_a \Longleftrightarrow \phi(\sigma_p) = a
\]
for all $a \in A - \{0\}$ and for all places $p$ of $\Q$ (including the infinite place $\infty$). Using equation (\ref{eInertiaPar}), one directly checks that
\[
\text{Par}(\text{Ev}(\phi))(\sigma) = \phi(\sigma)
\]
for $\sigma$ equal to $\sigma_p$ or $\sigma_\infty$. This implies that
\[
\text{Par}(\text{Ev}(\phi)) = \phi,
\]
since the $\sigma_p$ and $\sigma_\infty$ together form a set of topological generators of $\Gal(\Q^{\text{ab}}/\Q)$. A routine verification shows that $\text{Ev} \circ \text{Par} = \text{id}$, which completes the proof of the theorem.
\end{proof} 

The map $\text{Par}$ has two convenient properties. First of all, the space $\mathcal{A}$ is analytically easy to describe. Second of all, we have good control of the discriminant, which we make precise in our next theorem.

\begin{theorem}
\label{tDisc}
Let $\mathbf{v} = (v_a)_{a \in A - \{0\}} \in \mathcal{A}$ be such that $\textup{Par}(\mathbf{v}) \in \textup{Epi}(G_\Q, A)$. Then we have
\begin{align}
\label{eDisc}
v_p(\textup{Disc}(\textup{Par}(\mathbf{v}))) = v_p\left(\prod_{a \in A - \{0\}} v_a^{|A| \cdot \left(1 - \frac{1}{\textup{ord}(a)}\right)}\right)
\end{align}
for all primes $p$ coprime to $2 \cdot |A|$.
\end{theorem}

\begin{proof}
Let $\mathbf{v} = (v_a)_{a \in A - \{0\}} \in \mathcal{A}$ and let $p$ be a prime coprime to $2 \cdot |A|$. Write $\phi = \text{Par}(\mathbf{v})$ and write $L$ for the extension corresponding to $\phi$. Recall that the inertia subgroup $I_p$ of $\Gal(\Q^{\text{ab}}/\Q)$ is topologically generated by $\sigma_p$ (here we use that $p \neq 2$). Also recall that $I_p$ surjects on the inertia subgroup $I'_p$ of $\Gal(L/\Q)$.

Suppose that $\phi(\sigma_p) = a$. We compute that
\[
v_p\left(\prod_{a \in A - \{0\}} v_a^{|A| \cdot \left(1 - \frac{1}{\textup{ord}(a)}\right)}\right) = |A| \cdot \left(1 - \frac{1}{\textup{ord}(a)}\right),
\]
since $v_p(v_a) = 1$ and the entries of $\mathbf{v}$ are coprime. In order to compute the discriminant, write $K$ for the field corresponding to $\pi \circ \phi$, where $\pi: A \rightarrow A/\langle a \rangle$ is the natural quotient map. Then $p$ is unramified in the extension $K/\Q$, and all places of $K$ above $p$ are totally, tamely ramified in the cyclic extension $L/K$ of degree $\ord(a)$. Therefore we have
\[
v_p(\textup{Disc}(\textup{Par}(\mathbf{v}))) = |A| \cdot \left(1 - \frac{1}{\textup{ord}(a)}\right)
\]
as desired.
\end{proof}

It is most convenient to have a version of the above theorem that also deals with the wild places.

\begin{theorem}
\label{tWildDisc}
Let $\mathbf{v} = (v_a)_{a \in A - \{0\}} \in \mathcal{A}$ be such that $\textup{Par}(\mathbf{v}) \in \textup{Epi}(G_\Q, A)$. Then, for $p > 2$, $v_p(\textup{Disc}(\textup{Par}(\mathbf{v})))$ depends only on the order of $\textup{Par}(\mathbf{v})(\sigma_p)$.
\end{theorem}

\begin{proof}
This is completely a local question. Let $B$ be a finite abelian group and furthermore let $\phi \in \textup{Epi}(\Gal(\Q_p^{\text{ab}}/\Q_p), B)$. Then we have to show that the discriminant depends only on the order of $\phi(\sigma_p)$.

Write $K/\Q_p$ for the abelian extension corresponding to $\ker(\phi)$ by Galois theory. Denote by $I_{K/\Q_p}$ the inertia subgroup of $K/\Q_p$, which is equal to $\phi(I_p)$. We have an exact sequence
\[
0 \rightarrow I_{K/\Q_p} \rightarrow \Gal(K/\Q_p) \rightarrow \frac{\Gal(K/\Q_p)}{I_{K/\Q_p}} \rightarrow 0.
\]
After replacing $K$ by the compositum $KL'$ for $L'/\Q_p$ an unramified extension of degree equal to $|\Gal(K/\Q_p)|$, we observe that the above exact sequence is split. We may now filter $KL'$ as $KL'/M/\Q_p$, where $M$ is the unique subfield of $\Q_p(\zeta_{p^\infty})$ of degree $|I_{K/\Q_p}|$ followed by an unramified extension. One may therefore directly compute the discriminant of $M/\Q_p$. Then the theorem follows by several applications of the tower formula for the discriminant.
\end{proof}

\section{Criterion for weak approximation}
\label{sec:crit}
We recall the following results from \cite{FLN}. Let $A$ be a finite abelian group, and write $A^\vee := \text{Hom}(A, \Q/\Z)$ for the dual group. Recall that an $A$-extension is a surjective, continuous homomorphism from $G_\Q$ to $A$. We fix embeddings $G_{\Q_v} \rightarrow G_\Q$ for each place $v$ of $\Q$.

\begin{theorem}[{\cite[Theorem 6.2]{FLN}}]
Let $K/\Q$ be an $A$-extension. Then weak approximation holds if and only if the natural surjective map
\[
H^3(A, \Z)^\vee \rightarrow \Sha (R^1_{K/\Q}, \mathbb{G}_m)
\]
is an isomorphism.
\end{theorem}

\begin{theorem}[{\cite[Theorem 6.1]{FLN}}]
We have
\[
\Sha (R^1_{K/\Q}, \mathbb{G}_m)^\vee = \textup{ker}\left(H^3(A, \Z) \rightarrow \prod_v H^3(\textup{im}(G_{\Q_v}), \Z)\right).
\]
\end{theorem}

\noindent Combining these two theorems, we see that weak approximation holds if and only if
\[
H^3(A, \Z) = \textup{ker}\left(H^3(A, \Z) \rightarrow \prod_v H^3(\textup{im}(G_{\Q_v}), \Z)\right).
\]
Following \cite[Lemma 6.4]{FLN}, we see that the universal coefficient theorem gives canonical isomorphisms
\[
H^3(B, \Z) \cong \text{Ext}(H_2(B, \Z), \Z) \cong \text{Hom}(\wedge^2(B), \Q/\Z)
\]
for $B$ any finite abelian group. Then we get a diagram
\[ 
\begin{tikzcd}
H^3(A, \Z) \arrow{r}{\text{res}} \arrow{d} & H^3(\textup{im}(G_{\Q_v}), \Z) \arrow{d} \\%
\text{Hom}(\wedge^2(A), \Q/\Z) \arrow{r}{\varphi} & \text{Hom}(\wedge^2(\textup{im}(G_{\Q_v})), \Q/\Z)
\end{tikzcd}.
\]
The bottom map $\varphi$ in the above diagram is simply the following: consider the map $\textup{im}(G_{\Q_v}) \rightarrow A$. By functoriality of $\wedge^2$, we get a map $\wedge^2(\textup{im}(G_{\Q_v})) \rightarrow \wedge^2(A)$, which gives $\varphi$ after dualising. We summarise our discussion as follows.
\begin{theorem}
Weak approximation holds if and only if the natural restriction map
\[
\textup{Hom}(\wedge^2(A), \Q/\Z) \rightarrow \textup{Hom}(\wedge^2(\textup{im}(G_{\Q_v}), \Q/\Z)
\]
is the zero map for each place $v$.
\end{theorem}

We decompose $A$ as
\[
A = \bigoplus_{i = 1}^r \bigoplus_{j = 1}^{n_i} \Z/p_i^{e_j}\Z,
\]
where $p_1 < \dots < p_r$ are prime numbers, and $e_1, \dots, e_{n_i}$ are non-negative exponents with $e_1 \leq \dots \leq e_{n_i}$.

\begin{theorem}
\label{tLocalConditions}
There exists a subspace $\mathcal{S}$ of $A[p_1]$ with the following property. Suppose that $a \in A[p_1] - \{0\}$. Then for $b \in A$, we have that
\[
\textup{Hom}(\wedge^2(A), \Q/\Z) \rightarrow \textup{Hom}(\wedge^2(\langle a, b \rangle), \Q/\Z)
\]
is the zero map if and only if
\begin{itemize}
\item if $a \in \mathcal{S}$, then $b$ must be in a certain subset of $A/A[p_1]$ depending only on $a$;
\item if $a \not \in \mathcal{S}$, then $b \in p_1 A + \langle a \rangle$.
\end{itemize}
\end{theorem}

\begin{proof}
We take $\mathcal{S} = A[p_1] \cap p_1A$. Let us first suppose that $a \in \mathcal{S}$. Also recall that we have a canonical isomorphism
\[
\textup{Hom}(\wedge^2(A), \Q/\Z)\cong \{\text{alternating } \Z\text{-multilinear maps } A \times A \rightarrow \Q/\Z\}
\]
by the universal property of $\wedge^2$. Let $\varphi: A \times A \rightarrow \Q/\Z$ be alternating and let $t \in A[p_1]$. We claim that
\[
\varphi(a, t) = 0.
\]
Once the claim is proven, we immediately see that $\textup{Hom}(\wedge^2(A), \Q/\Z) \rightarrow \textup{Hom}(\wedge^2(\langle a, b \rangle), \Q/\Z)$ is the zero map if and only if $\textup{Hom}(\wedge^2(A), \Q/\Z) \rightarrow \textup{Hom}(\wedge^2(\langle a, b + t \rangle), \Q/\Z)$ is the zero map. Therefore the first part of the theorem follows from the claim.

Let us now prove the claim. Since $a \in p_1A$, we may take $b \in A$ such that $p_1 b = a$. We then have the equalities
\[
\varphi(a, t) = p_1 \varphi(b, t) = \varphi(b, p_1t) = \varphi(b, 0) = 0
\]
as claimed.

It remains to prove the second part of the theorem. To this end, let $a \not \in \mathcal{S}$. One readily verifies that the map $\textup{Hom}(\wedge^2(A), \Q/\Z) \rightarrow \textup{Hom}(\wedge^2(\langle a, b \rangle), \Q/\Z)$ is identically zero for $b \in p_1 A + \langle a \rangle$. Now suppose that $b \not \in p_1A + \langle a \rangle$. We must show that the map $\textup{Hom}(\wedge^2(A), \Q/\Z) \rightarrow \textup{Hom}(\wedge^2(\langle a, b \rangle), \Q/\Z)$ is not identically zero. Write $\Gamma := \langle a, b \rangle$. We claim that the injection
\[
0 \rightarrow \Gamma[p_1^\infty] \rightarrow A[p_1^\infty]
\]
is split. This is equivalent to the claim that
\[
p_1^k A[p_1^\infty] \cap \Gamma[p_1^\infty] = p_1^k \Gamma[p_1^\infty]
\]
for all $k > 0$. Observe that the inclusion $\supseteq$ is always true. 

Let $c \in p_1^k A[p_1^\infty] \cap \Gamma[p_1^\infty]$ and take $d$ such that $p_1^k d = c$. Then we have
\[
p_1^k d = n \cdot a + m \cdot b
\]
for some integers $n, m \in \Z$. In particular we deduce that $m \cdot b \in p_1 A + \langle a \rangle$. Since $A/(p_1A + \langle a \rangle)$ has exponent $p_1$ and since $b \not \in p_1 A + \langle a \rangle$ by assumption, it follows that $m$ must be divisible by $p_1$. Therefore we conclude that $n \cdot a$ is a multiple of $p_1$ in $A$. Because $a \not \in \mathcal{S}$, this forces that $p_1$ divides $n$ and thus $n \cdot a = 0$. We conclude that $c = p_1^k d = m \cdot b$. Because $b \not \in p_1A$, we deduce that $p_1^k \mid m$ and then $c \in p_1^k \Gamma[p_1^\infty]$ as desired.

We next claim that $\Gamma \otimes_{\Z} \FF_{p_1}$ has dimension $2$. Once the claim is proven, we deduce that $\wedge^2(\Gamma[p_1^\infty]) \cong \FF_{p_1}$. Using that the injection $0 \rightarrow \Gamma[p_1^\infty] \rightarrow A[p_1^\infty]$ is split, one readily verifies that the map
\[
\textup{Hom}(\wedge^2(A), \Q/\Z) \rightarrow \textup{Hom}(\wedge^2(\Gamma), \Q/\Z)
\]
is not identically zero. Therefore it is enough to establish the claim.

To prove the claim, we remark that the order of $b$ is divisible by $p_1$. Indeed, if not, we would have $b \in p_1A$ contrary to our assumptions. We have an exact sequence
\[
0 \rightarrow \langle a \rangle \cap \langle b \rangle \xrightarrow{\Delta} \langle a \rangle \oplus \langle b \rangle \xrightarrow{\Sigma} \langle a, b \rangle \rightarrow 0,
\]
where $\Delta$ is the map $x \rightarrow (x, -x)$ and $\Sigma$ is the sum map $(x, y) \mapsto x + y$. Therefore it suffices to show that $\langle a \rangle \cap \langle b \rangle = \{0\}$. If not, then there must be some $k \in \Z$ such that $k \cdot b = a$. Since $a \not \in \mathcal{S}$, it follows that $k$ is not divisible by $p_1$. But then we deduce that $b \in p_1 A + \langle a \rangle$, contrary to our assumptions. This gives the claim upon tensoring the above sequence by $\FF_{p_1}$.
\end{proof}

\section{Analytic tools}
\subsection{Sums of multiplicative functions}
After the various reduction steps in the later sections, the count in which we are interested will be expressed as a character sum. For $\ell > 2$, the main term will occur when the combinations of characters are principal. In such cases, we will repeatedly call upon the following general theorem of Koukoulopoulos \cite[Theorem 13.2]{Kou} based on the earlier work Granville--Koukoulopoulos~\cite{GK}.

\begin{theorem}[{\cite[Theorem 1]{GK}}]
\label{thm:GK}
Let $Q \geq 2$ be a parameter and let $f$ be a multiplicative function such that there exists $\alpha \in \mathbb{C}$ with
\begin{align}
\label{eOnPrimes}
\sum_{p \leq x} f(p) \log p = \alpha x + O_A\left( \frac{x}{(\log x)^A} \right) \quad (x \geq Q)
\end{align}
for all $A>0$. Moreover, suppose, for all $n$, that $\vert f(n) \vert \leq \tau_k(n)$ for some positive real number $k$. Fix $\epsilon > 0$ and $J \in \Z_{\geq 1}$. Then, we have 
\[
\sum_{n \leq x} f(n) =  x \sum_{j = 0}^{J - 1} c_j\frac{(\log x)^{\alpha - j - 1}}{\Gamma(\alpha - j)} + O\left(\frac{x (\log Q)^{2k + J - 1}}{(\log x)^{J + 1 - \textup{Re}(\alpha)}}\right)
\]
for $x \geq e^{(\log Q)^{1 + \epsilon}}$ and some explicit constants $c_j$. The implied constant depends at most on $k$, $J$, $\epsilon$ and the implied constant in equation (\ref{eOnPrimes}) for $A$ large enough in terms of $k$, $J$ and $\epsilon$ only. Furthermore, we have
\[
c_0 =  \prod_p \left( 1 + \frac{f(p)}{p} + \frac{f(p^2)}{p^2} + \cdots\right)\left(1 - \frac{1}{p} \right)^{\alpha}
\]
and $c_j \ll_{j, k} (\log Q)^{j + 2k}$.
\end{theorem}

\subsection{An abstract large sieve}
To handle some of the combinations of characters with large modulus which are non-principal, we will use a number field large sieve. Let $K$ be a number field and let $\ell$ be a prime number. If $\mathfrak{f}$ is an ideal, we write $S_{\mathfrak{f}}$ for the subset of $\alpha \in O_K$ coprime with $\mathfrak{f}$. We also write $N(w)$ for the absolute norm of an element. Let $M \geq 1$ be an integer. Suppose that we are given a map
\[
\gamma : S_{M O_K} \times S_{M O_K} \rightarrow \{0\} \cup \{\zeta_{\ell}^i : i = 0, \dots, \ell - 1\}
\]
and a subset $A_{\text{bad}}$ of $\Z_{\geq 0}$ satisfying the following properties:
\begin{enumerate}
\item[(P1)] we have $\gamma(w, z_1z_2) = \gamma(w, z_1) \gamma(w, z_2)$ and $\gamma(w_1w_2, z) = \gamma(w_1, z) \gamma(w_2, z)$; \quad{}\emph{(multiplicativity)}
\item[(P2)] if $z_1, z_2, w \in S_{M O_K}$ satisfy $z_1 \equiv z_2 \bmod N(w)$ and $z_1 \equiv z_2 \bmod M$, then we have $\gamma(w, z_1) = \gamma(w, z_2)$. Furthermore, if $N(w) \not \in A_{\text{bad}}$, then we have
\[
\sum_{\substack{\xi \bmod MN(w) \\ \gcd(\xi, M) = (1)}} \gamma(w, \xi) = 0; \quad{}\text{\emph{(periodicity)}}
\]
\item[(P3)] we have
\[
\sum_{\substack{n \in A_{\text{bad}} \\ n \leq X}} 1 \leq C_1 X^{1 - C_2}
\]
for some absolute constants $C_1 > 0$ and $0 < C_2 < 1$. \quad{}\emph{(bad count)}
\end{enumerate}

Decompose
\[
O_K^\ast = T \oplus V,
\]
where $T$ is torsion and $V$ is free. Such a decomposition is not unique, but we will fix one such decomposition. Fix a fundamental domain $\mathcal{D} \subseteq O_K$ as in \cite[Subsection 3.3]{KM} for the action of $V$ on $O_K$. We will recite the properties of the fundamental domain that we need.

\begin{lemma}
\label{lFund}
The fundamental domain $\mathcal{D} \subseteq O_K$ has the following properties
\begin{itemize}
\item for all $\alpha \in O_K$, there exists a unique $v \in V$ such that $v \alpha \in \mathcal{D}$. Furthermore, we have
\[
\{u \in O_K^\ast : u \alpha \in \mathcal{D}\} = \{\mu v : \mu \in T\};
\]
\item fix an integral basis $\eta = \{\eta_1, \dots, \eta_n\}$ of $O_K$. Then there exists a constant $C_\eta > 0$ such that $\alpha = a_1 \eta_1 + \dots + a_n \eta_n \in \mathcal{D}$ ($a_i \in \Z$) implies that $|a_i| \leq C_\eta \cdot N(\alpha)^{1/n}$.
\end{itemize}
\end{lemma}

\begin{proof}
This is \cite[Lemma 3.5]{KM}.
\end{proof}

We will consider bilinear sums of the type
\[
B(X, Y, \delta, \epsilon, t_1, t_2) = \sum_{\substack{w \in t_1 \mathcal{D}(X) \\ w \equiv \delta \bmod M}} \sum_{\substack{z \in t_2 \mathcal{D}(Y) \\ z \equiv \epsilon \bmod M}} \alpha_w \beta_z \gamma(w, z),
\]
where $(\alpha_w)_w$ and $(\beta_z)_z$ are sequences of complex numbers bounded in absolute value by $1$, $\delta$ and $\epsilon$ are invertible congruence classes modulo $M$, $t_1$ and $t_2$ are fixed elements of $T$ (so $t_i \mathcal{D}$ is a translate of the fundamental domain) and $X, Y \geq 2$ are real numbers. Here we use the notation $(t_i \mathcal{D})(X)$ for the subset of $\alpha \in t_i \mathcal{D}$ with $N(\alpha) \leq X$. 

\begin{proposition}
\label{pLargeSieve}
Assume that $X \leq Y$. Then we have
\[
|B(X, Y, \delta, \epsilon, t_1, t_2)| \ll \left(X^{\frac{-C_2}{3n}} + Y^{\frac{-1}{6n}}\right) XY (\log XY)^{C_K},
\]
where $n = [K : \Q]$ and $C_K$ is a constant depending only on $K$. The implied constant depends only on $K$, $M$ and the constants $C_1, C_2$.
\end{proposition}

Since $0 < C_2 < 1$, we achieve a power saving in both $X$ and $Y$. Careful scrutiny of the proof shows that the constant $C_K$ may be taken to depend at most on $n$. However, the same can not be said of the implicit constant, which will likely depend on the regulator of $K$ with the current argument. This will not be a cause for concern, however, as in our application this field will be a fixed cyclotomic field of the form $\Q(\zeta_\ell)$. 

\begin{remark}
This result should be compared with Heath-Brown's sieve for quadratic characters, see \cite[Corollary 4]{HB} or its number field analogues (e.g. \cite[Theorem 1.1]{GL} and \cite[Theorem 1.3]{Blomer}). There one considers bilinear sums of the Legendre symbol $\left( \frac{z}{w} \right)$ (or an analogous Hecke family), which naturally satisfies properties (P1) -- (P3). Heath-Brown is able to produce a greater power saving, however in his result the $(\log XY)^{C_{K,k}}$ term is replaced by $(XY)^{\epsilon}$. It will be crucial in our application that this term is at worst a fixed power of the logarithm.
\end{remark}

\begin{proof}[Proof of Proposition \ref{pLargeSieve}]
The argument is a minor generalisation of \cite[Proposition 3.6]{KM}. Pick an integer $k \geq 1$ that we will optimise later. We start the proof by applying H\"older's inequality to
\[
1 = \frac{k - 1}{k} + \frac{1}{k},
\]
which gives
\begin{align*}
|B(X, Y, \delta, \epsilon, t_1, t_2)|^k &\leq \left(\sum_{\substack{w \in t_1 \mathcal{D}(X) \\ w \equiv \delta \bmod M}} |\alpha_w| \cdot \left|\sum_{\substack{z \in t_2 \mathcal{D}(Y) \\ z \equiv \epsilon \bmod M}} \beta_z \gamma(w, z)\right|\right)^k \\
&\leq \left(\sum_{\substack{w \in t_1 \mathcal{D}(X) \\ w \equiv \delta \bmod M}} |\alpha_w|^{\frac{k}{k - 1}}\right)^{k - 1} \cdot \left(\sum_{\substack{w \in t_1 \mathcal{D}(X) \\ w \equiv \delta \bmod M}} \left|\sum_{\substack{z \in t_2 \mathcal{D}(Y) \\ z \equiv \epsilon \bmod M}} \beta_z \gamma(w, z)\right|^k\right) \\
&\ll X^{k - 1} \cdot \sum_{\substack{w \in t_1 \mathcal{D}(X) \\ w \equiv \delta \bmod M}} \left|\sum_{\substack{N(z) \leq Y^k}} \beta'_z \gamma(w, z)\right|,
\end{align*}
where
\[
\beta'_z := \sum_{\substack{z_1 \cdot \ldots \cdot z_k = z \\ z_i \in t_2 \mathcal{D}(Y) \\ z_i \equiv \epsilon \bmod M}} \beta_{z_1} \cdot \ldots \cdot \beta_{z_k}.
\]
Here we used property (P1) to expand the $k$-fold product. 

Fix an integral basis $\eta_1, \dots, \eta_n$ of $O_K$, where $n = [K : \Q]$. We call an element $z$ $(C, Y)$-well-balanced if
\begin{align}
\label{eIntegral}
z = a_1\eta_1 + \dots + a_n\eta_n, \quad a_i \in \Z
\end{align}
implies that $|a_i| \leq C Y^{1/n}$. From the construction of the fundamental domain $\mathcal{D}(Y)$, see the second property of Lemma \ref{lFund}, it follows that there exists a constant $C > 0$ depending only on $k$, $K$ and the choice of integral basis such that $\beta'_z = 0$ if $z$ is not $(C, Y^k)$-well-balanced. Fix such a choice of absolute constant $C$. Then, we may assume from now on that we are summing over all $z$ such that $N(z) \leq Y^k$ and $z$ is $(C, Y^k)$-well-balanced. Write $\mathcal{B}(Y, C)$ for the set of $z \in O_K$ such that $|a_i| \leq CY^{1/n}$ upon expanding $z$ as in equation (\ref{eIntegral}) and such that $z$ is coprime with $M$. For the remainder of this proof, all our implied constants may depend on $K$, $M$, $k$, $C_1$, $C_2$, our choice of $C$ and our choice of integral basis. We rewrite
\begin{align*}
\sum_{\substack{w \in t_1 \mathcal{D}(X) \\ w \equiv \delta \bmod M}} \left|\sum_{\substack{N(z) \leq Y^k \\ z \in \mathcal{B}(Y^k, C)}} \beta'_z \gamma(w, z)\right| 
&= \sum_{\substack{w \in t_1 \mathcal{D}(X) \\ w \equiv \delta \bmod M}} \epsilon(w) \sum_{\substack{N(z) \leq Y^k \\ z \in \mathcal{B}(Y^k, C)}} \beta'_z \gamma(w, z) \\
&= \sum_{\substack{N(z) \leq Y^k \\ z \in \mathcal{B}(Y^k, C)}} \beta'_z \sum_{\substack{w \in t_1 \mathcal{D}(X) \\ w \equiv \delta \bmod M}} \epsilon(w) \gamma(w, z),
\end{align*}
where $\epsilon(w)$ are complex numbers of absolute value $1$. We now drop the condition $N(z) \leq Y^k$. This does not change the sum, since $\beta'_z = 0$ if $N(z) > Y^k$. Then, the Cauchy--Schwarz inequality yields
\begin{multline*}
|B(X, Y, \delta, \epsilon, t_1, t_2)|^{2k} \ll \\
X^{2k - 2} \cdot \left(\sum_{\substack{z \in \mathcal{B}(Y^k, C)}} |\beta'_z|^2\right) \cdot \left(\sum_{\substack{z \in \mathcal{B}(Y^k, C)}} \sum_{\substack{w_1 \in t_1 \mathcal{D}(X) \\ w_1 \equiv \delta \bmod M}} \sum_{\substack{w_2 \in t_1 \mathcal{D}(X) \\ w_2 \equiv \delta \bmod M}} \epsilon(w_1) \epsilon(w_2) \gamma(w_1w_2, z)\right)
\end{multline*}
thanks to property (P1). We bound the former sum by
\[
\sum_{\substack{z \in \mathcal{B}(Y^k, C)}} |\beta'_z|^2 \ll Y^k (\log Y)^{C_{K, k}},
\]
where $C_{K, k}$ is an effectively computable constant depending only on $K$ and $k$. For the latter sum we invert the order of summation to get
\[
\sum_{\substack{w_1 \in t_1 \mathcal{D}(X) \\ w_1 \equiv \delta \bmod M}} \sum_{\substack{w_2 \in t_1 \mathcal{D}(X) \\ w_2 \equiv \delta \bmod M}} \epsilon(w_1) \epsilon(w_2) \sum_{z \in \mathcal{B}(Y^k, C)} \gamma(w_1w_2, z).
\]
We have the estimates
\[
\sum_{z \in \mathcal{B}(Y^k, C)} \gamma(w_1w_2, z) \ll
\left\{
	\begin{array}{ll}
		Y^k  & \mbox{if } N(w_1w_2) \in A_{\textup{bad}} \\
		\sum_{i = 1}^n X^{2i} Y^{k(1 - \frac{i}{n})} & \mbox{if } N(w_1w_2) \not \in A_{\textup{bad}}.
	\end{array}
\right.
\]
Indeed, the first inequality is just the trivial bound. For the second inequality, we use property (P2) and split $\mathcal{B}(Y^k, C)$ into boxes of side length $MN(w_1w_2) \leq MX^2$. 

From now on we shall take $k \geq 2n$. Hence we have $Y^{k/n} \geq X^2$, since we assumed that $Y \geq X$. Therefore the last bound can be simplified to $X^2 Y^{k(1 - 1/n)}$. Thanks to property (P3), we get the bound
\[
\sum_{\substack{w_1 \in t_1 \mathcal{D}(X) \\ w_1 \equiv \delta \bmod M}} \sum_{\substack{w_2 \in t_1 \mathcal{D}(X) \\ w_2 \equiv \delta \bmod M}} \mathbf{1}_{N(w_1w_2) \in A_{\text{bad}}} \ll_{K, k, C_1} X^{2 - 2C_2} (\log X)^{C_{K, k}}
\]
for a potentially different effectively computable constant $C_{K, k}$ depending only on $K$ and $k$. This shows that
\begin{multline*}
\sum_{\substack{w_1 \in t_1 \mathcal{D}(X) \\ w_1 \equiv \delta \bmod M}} \sum_{\substack{w_2 \in t_1 \mathcal{D}(X) \\ w_2 \equiv \delta \bmod M}} \epsilon(w_1) \epsilon(w_2) \sum_{z \in \mathcal{B}(Y^k, C)} \gamma(w_1w_2, z) \ll \\
\left(X^{2 - 2C_2} Y^k + X^4 Y^{k(1 - 1/n)}\right) (\log XY)^{C_{K, k}}.
\end{multline*}
We conclude that
\[
|B(X, Y, \delta, \epsilon, t_1, t_2)|^{2k} \ll \left(X^{2k - 2C_2} Y^{2k} + X^{2k + 2} Y^{2k - \frac{k}{n}}\right) (\log XY)^{C_{K, k}}.
\]
We take $k = 3n$, which depends only on $K$, to finish the proof of the proposition.
\end{proof}

\section{Reduction to multicyclic extensions}
\label{sec:reduction}
The aim of this section is to reduce our main theorem to the case of multicyclic extensions. This is not quite possible, but instead we may reduce to multicyclic extensions, where the decomposition group condition is replaced by a splitting condition depending on $p \bmod M$, where we think of $M$ as being (almost) fixed. We formalise these type of conditions in our next definition. Write $\mathcal{P}(X)$ for the power set of a set $X$.

\begin{mydef}
Let $A = \mathbb{F}_\ell^n$ for some prime $\ell$ and integer $n > 0$. Write $\mathcal{S}$ for the set of subspaces of $A$ of dimension at most $1$. Also write $\mathbb{P}$ for the set of prime numbers. Let $M \in \Z_{\geq 1}$ and let $B \subseteq A$ be a subspace. Let $f: \mathbb{P} \times \mathcal{S} \rightarrow \sqcup_{S \in \mathcal{S}} \mathcal{P}(A/S)$ be a function satisfying: 
\begin{itemize}
\item we have $f(p, S) \subseteq A/S$ for all $(p, S)$. Furthermore, if $S \cap B = \{0\}$, then $f(p, S)$ is a coset of $B$ viewed as subgroup of $A/S$. If $S \cap B \neq \{0\}$, then $f(p, S) = A/S$;
\item for every fixed $S \in \mathcal{S}$, the function $p \mapsto f(p, S)$ depends only on $p \bmod M$. 
\end{itemize}
Let $g: \{p \in \mathbb{P} : p \mid M\} \rightarrow \{\emptyset \subset F \subseteq \mathbb{F}_\ell^n\}$. We say that the pair $(f, g)$ is a \emph{congruence function for $M$}. Moreover, let $K$ be an $A$-extension of $\Q$ given by $\phi \in \textup{Epi}(G_\Q, A)$. If $p$ is a prime, we write $K_p$ for the extension of $\Q$ corresponding to $G_\Q \xrightarrow{\phi} A \rightarrow A/\phi(I_p)$. We say that $K$ has the correct Frobenius elements for $(f, g)$ (abbreviated as $(f, g)$-correct or simply $f$-correct) if
\begin{itemize}
\item we have
\[
\textup{Frob}_{K_p/\Q}(p) \in f(p, \phi(I_p))
\]
for all $p$ dividing the discriminant of $K$ such that $p$ is coprime to $2M$;
\item we have that all $p \mid M$ are unramified in $K$ and that
\[
\textup{Frob}_{K/\Q}(p) \in g(p)
\]
for all $p$ dividing $M$, assuming that $p \neq 2$ or $\ell \neq 2$.
\end{itemize}
\end{mydef}

Observe that $\phi(I_p)$ can have dimension greater than $1$ only if $p = \ell = 2$. 

In the coming sections we will show the following result. We set $d(\ell)$ equal to $16$ if $\ell = 2$ and $1$ if $\ell > 2$. Let $\Delta = \Delta_{\ell}$ be the unique strongly multiplicative function satisfying $\Delta(p) = p$ for $p \neq \ell$ and $\Delta(\ell) = \ell^2$.

\begin{theorem}
\label{tMultiReduction}
Let $C > 0$. Let $A = \mathbb{F}_\ell^n$ for some prime $\ell$ and integer $n > 0$ and let $B$ be a subgroup of $A$. Then there exists $\delta > 0$ such that the following holds. Let $(f, g)$ be a congruence function for $M \in \Z_{\geq 1}$. Let $(c_a)_{a \in A - \{0\}}$ be a vector of integers. Let $G$ be a subgroup of $(\Z/M\Z)^\ast$. Let $H_a$ be a union of cosets of $G$ such that $|H_a| = |H_{a'}|$ for all $a, a' \in A - B$. Assume that $M \leq (\log X)^C$ and that $[(\Z/M\Z)^\ast : G] \leq C$. Let $S$ be a subset of $A - \{0\}$. Further assume that
\begin{align}
\label{eCongruenceFurtherFactor}
f(p, \langle a \rangle) = f(p', \langle a \rangle)
\end{align}
if $p \bmod M$ and $p' \bmod M$ are in the same coset of $G$. Then there is $C_{\textup{lead}} \geq 0$ such that for all real numbers $X \geq 100$
\[
\sum_{\substack{\mathbf{v} = (v_a)_{a \in A - \{0\}} \in \mathcal{A} \\ \prod_{a \in A - \{0\}} \Delta(|v_a|) \leq X \\ v_a \neq 1 \ \forall a \in S \\ v_a \equiv c_a \bmod d(\ell) \\ \gcd(v_a, M) = 1 \\ p \mid v_a \Rightarrow p \bmod M \in H_a}} \mathbf{1}_{\textup{Par}(\mathbf{v}) \ f\textup{-correct}} = C_{\textup{lead}} X (\log X)^{\alpha-1} + O\left(X (\log X)^{\alpha -1- \delta}\right),
\]
where
\[
\alpha =  \sum_{a \in B - \{0\}} \frac{|\textup{Lift}(H_a)|}{\varphi(\textup{lcm}(M, \ell))} + \sum_{a \in A - B} \frac{|B|}{\ell^{n - 1}} \cdot \frac{|\textup{Lift}(H_a)|}{\varphi(\textup{lcm}(M, \ell))}.
\]
Here $\textup{Lift}(G)$ denotes the subset of $(\Z/\textup{lcm}(M, \ell)\Z)^\ast$ consisting of elements that map to $1$ in $(\Z/\ell\Z)^\ast$ and to an element of $G$ in $(\Z/M\Z)^\ast$. The implied constant depends only on $C$ and $A$.

Furthermore, the leading constant $C_{\textup{lead}}$ does not depend on $S$, and there exists a constant $C_{\textup{max}}$, depending only on $\ell$ and $n$, such that $C_{\textup{lead}} \leq C_{\textup{max}}$. We have $C_{\textup{lead}} > 0$ if $\ell > 2$ or
\begin{itemize}
\item $f(p, S) = B$ for all $p$ coprime to $M$ and all $S$ such that $S \cap B = \{0\}$;
\item $0 \in g(p)$ for all $p \mid M$;
\item $c_a = 1$ for all $a \in A - \{0\}$.
\end{itemize}
\end{theorem}

In our application, the $A$ in the theorem will be the $\ell$-torsion subgroup of a fixed abelian group $A$. In light of the work of Section \ref{sec:crit}, we will take $B$ to be $A[\ell] \cap \ell A$. Then, in order to study $A$-extensions failing to have the weak approximation condition, we must ensure that Frobenius lands in one of the acceptable classes, as specified in Theorem \ref{tLocalConditions}. This is the purpose of the notion of congruence functions and $f$-correctness. The conditions on primes dividing $v_a$ allow us to control the splitting behaviour of primes which do not ramify in the extension from $\Q$ to a fixed field of $A/A[\ell]$ but do ramify in the remaining $A[\ell]$-extension.

\begin{remark}
This theorem is in many ways best possible. Further uniformity in $M$ is plausible but likely out of reach given the current state of knowledge regarding zeroes of $L$-functions. Perhaps the most interesting condition in the theorem is that the $H_a$ all have the same cardinality for $a \in A - B$. Indeed, the theorem is no longer true if the $H_a$ are allowed to be of arbitrary cardinality. More precisely, the true log exponent may be bigger than $\alpha-1$.

Consider for example the situation that $B = \{0\}$ and $f(p, S)$ is always the zero element of $A/S$. Fix some $a \in A - \{0\}$. If one takes $|H_b|$ very small for $b \neq a$ and $|H_a|$ very large, then one can obtain a larger log exponent by taking $v_b = 1$ for $b \neq a$ and letting $v_a$ vary freely.  Note that this example critically depends both on the shape of the congruence function and the sizes of $|H_a|$.
\end{remark}

We will now show how to derive our main theorem from Theorem \ref{tMultiReduction}. It is important to remark that the exponents appearing in the discriminant, equation \eqref{eDisc}, are not equal. In fact, the exponent is minimised by taking $a$ to be in $A[\ell]$, where $\ell$ is the smallest prime divisor of $A$. In the language of \cite{Wood}, the discriminant is not a fair counting function. It is precisely for this reason that we may reduce to multicyclic extensions, which is certainly not possible for fair counting functions.

Write $\mathcal{B}$ for those tuples in $\mathcal{A}$ that give rise to an element in $\text{Epi}(G_\Q, A)$. 

\begin{theorem}
\label{tMain}
Let $A$ be a non-trivial, finite, abelian group. Write $\ell$ for the smallest prime divisor of $A$. Then there exists $C_{\textup{weak}} > 0$ and $\delta > 0$ such that for all $X \geq 100$
\[
\sum_{\substack{\mathbf{v} = (v_a)_{a \in A - \{0\}} \in \mathcal{B} \\ \textup{Disc}(\textup{Par}(\mathbf{v})) \leq X}} \mathbf{1}_{\textup{W.A. holds}} = C_{\textup{weak}} X^{\frac{\ell}{|A| \cdot (\ell - 1)}} (\log X)^{\alpha(A)-1} + O\left(X^{\frac{\ell}{|A| \cdot (\ell - 1)}} \cdot (\log X)^{\alpha(A) - 1- \delta}\right),
\]
where 
\[
\alpha(A) =  \sum_{a \in A[\ell] - \{0\}} \frac{|\{b \in A : \textup{Hom}(\wedge^2(A), \Q/\Z) \rightarrow \textup{Hom}(\wedge^2(\langle a, b \rangle), \Q/\Z) \textup{ is zero}\}|}{(\ell - 1) \cdot |A|}.
\]
The implied constant depends only on $A$.
\end{theorem}

\begin{proof}
Fix a large number $C_3$. We split the sum as
\begin{multline}
\label{eSplit}
\sum_{\substack{\mathbf{v} = (v_a)_{a \in A - \{0\}} \in \mathcal{B} \\ \textup{Disc}(\textup{Par}(\mathbf{v})) \leq X}} \mathbf{1}_{\textup{W.A. holds}} = \\
\sum_{\substack{\mathbf{v} = (v_a)_{a \in A - \{0\}} \in \mathcal{B} \\ \textup{Disc}(\textup{Par}(\mathbf{v})) \leq X \\ \prod_{a \not \in A[\ell]} v_a > (\log X)^{C_3}}} \mathbf{1}_{\textup{W.A. holds}} + \sum_{\substack{\mathbf{v} = (v_a)_{a \in A - \{0\}} \in \mathcal{B} \\ \textup{Disc}(\textup{Par}(\mathbf{v})) \leq X \\ \prod_{a \not \in A[\ell]} v_a \leq (\log X)^{C_3}}} \mathbf{1}_{\textup{W.A. holds}}.
\end{multline}
The main term will come from the latter sum in equation (\ref{eSplit}). We will start by bounding the former sum. It will be convenient to set
\[
e_a := |A| \cdot \left(1 - \frac{1}{\text{ord}(a)}\right).
\]
Note that $e_a$ is minimal for $a \in A[\ell] - \{0\}$. We bound the former sum in equation (\ref{eSplit}) by
\begin{align}
\label{eFirstBound}
\sum_{\substack{\mathbf{v} = (v_a)_{a \in A - \{0\}} \in \mathcal{B} \\ \textup{Disc}(\textup{Par}(\mathbf{v})) \leq X \\ \prod_{a \not \in A[\ell]} v_a > (\log X)^{C_3}}} 1 
\leq \sum_{\substack{(w_a)_{a \in A - A[\ell]} \\ w_a \in \Z_{\geq 1} \\ \prod_{a \not \in A[\ell]} w_a > (\log X)^{C_3}}} \sum_{\substack{(v_a)_{a \in A[\ell] - \{0\}} \\ v_a \text{ squarefee coprime} \\ p \mid v_a \Rightarrow p \equiv 0, 1 \bmod \ell \\ \prod_{a \in A[\ell] - \{0\}} v_a^{|A| \cdot \left(1 - \frac{1}{\ell}\right)} \leq \frac{X}{\prod_{a \in A - A[\ell]} w_a^{e_a}}}} 1.
\end{align}
Writing $|A[\ell]| = \ell^n$, classical analytic number theory shows that
\[
\sum_{\substack{(v_a)_{a \in A[\ell] - \{0\}} \\ v_a \text{ squarefree coprime} \\ p \mid v_a \Rightarrow p \equiv 0, 1 \bmod \ell \\ \prod_{a \in A[\ell] - \{0\}} v_a^{|A| \cdot \left(1 - \frac{1}{\ell}\right)} \leq \frac{X}{\prod_{a \in A - A[\ell]} w_a^{e_a}}}} 1 
\ll_\ell \left(\frac{X}{\prod_{a \in A - A[\ell]} w_a^{e_a}}\right)^{\frac{\ell}{|A| \cdot (\ell - 1)}} (\log X)^{\frac{\ell^n - 1}{\ell - 1} - 1}.
\]
Plugging this in equation (\ref{eFirstBound}) yields the bound
\[
X^{\frac{\ell}{|A| \cdot (\ell - 1)}} \cdot (\log X)^{\frac{\ell^n - 1}{\ell - 1} - 1} \sum_{\substack{(w_a)_{a \in A - A[\ell]} \\ w_a \in \Z_{\geq 1} \\ \prod_{a \not \in A[\ell]} w_a > (\log X)^{C_3}}} \frac{1}{\prod_{a \in A - A[\ell]} w_a^{\frac{\ell \cdot e_a}{|A| \cdot (\ell - 1)}}}.
\]
Since $e_a$ is minimal when $a \in A[\ell]$, we see that the sum
\[
\sum_{\substack{(w_a)_{a \in A - A[\ell]} \\ w_a \in \Z_{\geq 1}}} \frac{1}{\prod_{a \in A - A[\ell]} w_a^{\frac{\ell \cdot e_a}{|A| \cdot (\ell - 1)}}}
\]
converges, because the exponent of $w_a$ is bigger than $1$. Therefore, the tail can be bounded by
\[
X^{\frac{\ell}{|A| \cdot (\ell - 1)}} \cdot (\log X)^{\frac{\ell^n - 1}{\ell - 1} - 1}\sum_{\substack{(w_a)_{a \in A - A[\ell]} \\ w_a \in \Z_{\geq 1} \\ \prod_{a \not \in A[\ell]} w_a > (\log X)^{C_3}}} \frac{1}{\prod_{a \in A - A[\ell]} w_a^{\frac{\ell \cdot e_a}{|A| \cdot (\ell - 1)}}}
\ll 
X^{\frac{\ell}{|A| \cdot (\ell - 1)}} \cdot (\log X)^{\frac{\ell^n - 1}{\ell - 1} - 1 - \delta C_3}
\]
for some $\delta > 0$ depending only on $A$. This is negligible, provided that we pick $C_3$ sufficiently large in terms of $A$. It remains to deal with the latter sum in equation (\ref{eSplit}).

\subsection*{The main term: the reduction step}
To prepare for our application of Theorem \ref{tMultiReduction}, we start by fixing all the variables $v_a$ with $a \in A - A[\ell]$. The induced homomorphism $\varphi: G_\Q \rightarrow A \rightarrow A/A[\ell]$ depends only on the $v_a$ with $a \in A - A[\ell]$. Define for $a \in A$ the set
\[
\text{Adm}(a) := \{b \in A : \textup{Hom}(\wedge^2(A), \Q/\Z) \rightarrow \textup{Hom}(\wedge^2(\langle a, b \rangle), \Q/\Z) \textup{ is zero}\}.
\]
Write $\widetilde{\pi}: A \rightarrow A/A[\ell]$ for the natural quotient map. We assume that
\begin{enumerate}
\item[(P1)] the homomorphism $\varphi$ is surjective. Write $K$ for the fixed field of $\varphi$, which is an $A/A[\ell]$-extension of $\Q$;
\item[(P2)] for all $p \mid v_a$ we have $\varphi(\text{Frob}_p) \in \widetilde{\pi}(\text{Adm}(a))$. Here and from now on we have made a fixed choice of a prime ideal $\mathfrak{p}$ of $\overline{\Z}$ above $p$ and we have also made a fixed choice of an element $\text{Frob}_p \in \{\sigma : \sigma(\mathfrak{p}) = \mathfrak{p}\} \subseteq \Gal(\overline{\Q}/\Q)$ lifting the Frobenius automorphism of $\overline{\Z}/\mathfrak{p} \cong \overline{\mathbb{F}_p}$;
\item[(P3)] there exists an $A$-extension of $\Q$ containing $K$ that satisfies weak approximation. We remark that this condition in fact implies (P1) and (P2).
\end{enumerate}
If these conditions are not met, we may freely ignore the variables $v_a$, since they do not contribute to the counting function. There exists at least one tuple $v_a$ satisfying (P1), (P2) and (P3) by an application of \cite[Proposition 5.5]{FLN} with $k = \Q$ and $G = A$. Indeed, weak approximation certainly holds if all decomposition groups are cyclic. We will use this later on to guarantee that our leading constant $C_{\text{weak}}$ is strictly greater than $0$.

We will now work towards applying Theorem \ref{tMultiReduction}. We take $M = \mathfrak{f}(K)^2$, where $\mathfrak{f}(K)$ is the conductor of $K$. There exists a subgroup $G$ of $(\Z/M\Z)^\ast$ such that $p$ splits completely in $K$ if and only if $p \bmod M \in G$. Then we take for $a \in A[\ell] - \{0\}$
\[
H_a := \bigcup_{b \in \widetilde{\pi}(\text{Adm}(a))} \text{FrobInv}(b) \cdot G,
\]
where $\text{FrobInv}(b)$ is any prime $p$, not dividing $M$, with Artin symbol in $K$ equal to $b$. By choosing $C$ sufficiently large in terms of the abelian group $A$, we see that the condition $[(\Z/M\Z)^\ast : G] \leq C$ in Theorem \ref{tMultiReduction} holds.

We take $B = A[\ell] \cap \ell A$. We have to check that $|H_a| = |H_{a'}|$ for all $a, a' \in A[\ell] - B$. This is equivalent to showing that
\begin{align}
\label{epiCard}
|\widetilde{\pi}(\text{Adm}(a))| = |\widetilde{\pi}(\text{Adm}(a'))|
\end{align}
for all $a, a' \in A[\ell] - B$. If $\langle a \rangle = \langle a' \rangle$, then the result is correct. So now suppose that $\langle a \rangle \neq \langle a' \rangle$. We claim that there exists an automorphism $\psi$ of $A$ such that $\psi(a) = a'$. 

Let us first show that the claim implies equation (\ref{epiCard}). To this end, take an automorphism $\psi$ that sends $a$ to $a'$. Consider the induced map $\psi^\ast: \text{Adm}(a) \rightarrow \text{Adm}(a')$, which sends $b$ to $\psi(b)$. To check that this is well-defined, we further claim that
\[
\textup{Hom}(\wedge^2(A), \Q/\Z) \rightarrow \textup{Hom}(\wedge^2(\langle a, b \rangle), \Q/\Z) \textup{ is zero}
\]
if and only if
\[
\textup{Hom}(\wedge^2(A), \Q/\Z) \rightarrow \textup{Hom}(\wedge^2(\langle \psi(a), \psi(b) \rangle), \Q/\Z) \textup{ is zero}.
\]
The former statement is equivalent to all alternating maps $\lambda: A \times A \rightarrow \Q/\Z$ satisfying $\lambda(a, b) = 0$, while the latter statement is equivalent to all alternating maps $\lambda: A \times A \rightarrow \Q/\Z$ satisfying $\lambda(\psi(a), \psi(b)) = 0$. As $\lambda$ runs through all alternating maps, so does $\lambda(\psi(-), \psi(-))$ and therefore the above statements are all equivalent. This shows that $\psi^\ast$ is well-defined. It is now readily verified that $\psi^\ast$ is a bijection. We next observe that $\text{Adm}(a)$ is a subgroup of $A$. Therefore, to establish equation (\ref{epiCard}), it suffices to show that $|\text{Adm}(a) \cap A[\ell]| = |\text{Adm}(a') \cap A[\ell]|$. But this is true because $\psi$ is an automorphism that restricts to an isomorphism between $\text{Adm}(a)$ and $\text{Adm}(a')$. We conclude that the claim indeed implies equation (\ref{epiCard}).

Let us now prove the claim. If the injection
\[
0 \rightarrow \langle a, a' \rangle \rightarrow A
\]
is split, then it is straightforward to construct the desired bijection $\psi$. So now suppose that the above injection is not split. Because $a, a' \not \in B$, this means that there exist $n, m \in \Z$ both not divisible by $\ell$ such that $na + ma' \in \ell A$. Using that $a \not \in B$, we may decompose $A$ as an internal direct sum $\langle a \rangle \oplus C$ for some subgroup $C$ of $A$. Now consider the homomorphism $\psi$ that is the identity on $C$ and sends $a$ to $a'$.

It remains to prove that $\psi$ is surjective. By construction we have that $C \subseteq \text{im}(\psi)$ and that $a' \in \text{im}(\psi)$. Therefore it suffices to show that $a \in \text{im}(\psi)$. By the relation $na + ma' \in \ell A$ and the inclusion $\ell A \subseteq C$, we deduce that $na \in \text{im}(\psi)$. Because $n$ is coprime to $\ell$, we conclude that $a \in \text{im}(\psi)$. We have now proven the claim and thus equation (\ref{epiCard}).

We will now construct a congruence function $(f, g)$ for $M$.
Suppose that $p \neq 2$ is unramified in $K$. Then, if $p$ ramifies in $\text{Par}(\mathbf{v})$, we have $p \mid v_a$ for some $a \in A[\ell] - \{0\}$. Take $S = \langle a \rangle \in \mathcal{S}$, which has dimension $1$. For weak approximation to hold, we certainly must have that $\text{Frob}_{K/\Q}(p) \in \widetilde{\pi}(\text{Adm}(a))$ (or equivalently $p \bmod M \in H_a$), which we will assume from now on. We will now distinguish two cases.

\subsection*{Case I: $a \in B$}
In this case, Theorem \ref{tLocalConditions} yields that $\text{Frob}_p$ automatically lands in $\text{Adm}(a)$. Correspondingly, we take $f(p, S) = A/S$.

\subsection*{Case II: $a \not \in B$}
In this case Theorem \ref{tLocalConditions} tells us that $\text{Adm}(a) = \ell A + \langle a \rangle$. Define $\mathbf{x}_Q$ to be the unique tuple that equals $v_a$ for $a \in A - A[\ell]$ and equals $1$ for $a \in A[\ell] - \{0\}$. We claim that there exists an element $h(p, S) \in A[\ell]$ such that
\begin{align}
\label{eTwistingh}
\text{Par}(\mathbf{x}_Q)(\text{Frob}_p) + h(p, S) \in \text{Adm}(a).
\end{align}
By construction of $\mathbf{x}_Q$ we have that $\widetilde{\pi} \circ \text{Par}(\mathbf{x}_Q) = \varphi$. Therefore it follows from $\text{Frob}_{K/\Q}(p) \in \widetilde{\pi}(\text{Adm}(a))$ that equation (\ref{eTwistingh}) is true after applying $\widetilde{\pi}$, establishing the claim. Then we take $f(p, S)$ to be the coset $h(p, S) + B$. One may now directly verify that condition (\ref{eCongruenceFurtherFactor}) is satisfied.

\subsection*{The bad primes}
Next consider a prime $p \neq 2$ that divides the conductor of $K$, which implies that $p \mid v_a$ for some $a \not \in A[\ell]$. Let $g(p)$ be the subset of $a \in A[\ell]$ satisfying
\[
\text{Par}(\mathbf{x}_Q)(\text{Frob}_p) + a \in \text{Adm}(a).
\]
We deduce from property (P2) of $K$ that $g(p)$ is non-empty.

\subsection*{Completing the reduction step}
We now rewrite
\[
\sum_{\substack{\mathbf{v} = (v_a)_{a \in A - \{0\}} \in \mathcal{B} \\ \textup{Disc}(\textup{Par}(\mathbf{v})) \leq X \\ \prod_{a \not \in A[\ell]} v_a \leq (\log X)^{C_3}}} \mathbf{1}_{\text{W.A. holds}}
\]
as
\[
\sum_{\substack{(v_a)_{a \in A - A[\ell]} \\  \\ \prod_{a \not \in A[\ell]} v_a \leq (\log X)^{C_3}}} \sum_{(c_a)_{a \in A[\ell] - \{0\}} \in (\Z/d(\ell)\Z)^{A[\ell] - \{0\}}} \mathbf{1}_{(c_a) \text{ nice}} \sum_{\substack{\mathbf{w} = (w_a)_{a \in A[\ell] - \{0\}} \\ \mathbf{v} \in \mathcal{B} \\ \textup{Disc}(\textup{Par}(\mathbf{v})) \leq X \\ w_a \equiv c_a \bmod d(\ell) \\ p \mid v_a \Rightarrow p \bmod M \in H_a}} \mathbf{1}_{\text{W.A. holds}},
\]
where $\mathbf{v}$ is the tuple obtained by concatenating $(v_a)$ and $(w_a)$ and where we will soon define when $(c_a)$ is nice.

Fix the tuple $\mathbf{v}_Q = (v_a)_{a \in A - A[\ell]}$ satisfying the assumptions (P1), (P2) and (P3) and fix a tuple $(c_a)_{a \in A[\ell] - \{0\}}$. We first claim that the condition $\mathbf{1}_{\text{W.A. holds}}$ may be replaced by $\mathbf{1}_{\text{Par}(\mathbf{w}) \ f\text{-correct}}$. First of all we remark that $D_p$ is certainly cyclic if $p$ is unramified in $\text{Par}(\mathbf{v})$. At the odd ramified places this is true by construction of $f$. Finally, whether the natural restriction map $\textup{Hom}(\wedge^2(A), \Q/\Z) \rightarrow \textup{Hom}(\wedge^2(\textup{im}(G_{\Q_2}), \Q/\Z)$ is zero, is entirely determined by $\mathbf{v}_Q = (v_a)_{a \in A - A[\ell]}$ and $(c_a)_{a \in A[\ell] - \{0\}}$. Indeed, if $\ell > 2$, then this map is always zero as $2$ is necessarily unramified in $A$. If $\ell = 2$, then this follows from the fact that the restrictions of $\psi_{v_a, 2, 1}$ and $\psi_{v_b, 2, 1}$ to $G_{\Q_2}$ are equal if $v_a \equiv v_b \bmod 16$ (recall that $d(2) = 16$) combined with equation (\ref{eFundPar}). Now we simply define $(c_a)$ to be nice if the above map is zero.

Let us now explicate the condition $\mathbf{v} = (v_a)_{a \in A - \{0\}} \in \mathcal{B}$ in terms of the variables in $A[\ell] - \{0\}$. One directly checks that
\[
\mathbf{v} \in \mathcal{B} \Longleftrightarrow \langle a \in A - \{0\} : v_a \neq 1 \rangle = A
\]
for all $\mathbf{v} \in \mathcal{A}$. Therefore, since $\varphi$ is already assumed to be surjective by (P1), we have that there exists a subspace $T$, containing $B$ and depending on $\mathbf{v}_Q = (v_a)_{a \in A - A[\ell]}$, such that
\begin{align}
\label{eSurjCondition}
\mathbf{v} \in \mathcal{B} \Longleftrightarrow T + \langle a \in A[\ell] - \{0\} : v_a \neq 1 \rangle = A[\ell].
\end{align}

Next we compare $\text{Disc}(\text{Par}(\mathbf{v}))$ with $\text{Disc}(\text{Par}(\mathbf{w}))$. To do so, we compute the $p$-adic valuation of the discriminant using Theorem \ref{tDisc} for primes $p \nmid 2|A|$. Locally at $2$, we have just proven that the restriction of $\text{Par}(\mathbf{v})$ is unramified if $\ell > 2$ and completely determined by $c_a$ and $\mathbf{v}_Q = (v_a)_{a \in A - A[\ell]}$ if $\ell = 2$. For the odd places $p \mid |A|$, we compute the $p$-adic valuation of the discriminant using Theorem \ref{tWildDisc}. More precisely, if $\text{Par}(\sigma_p)$ is the identity, then $v_p(\text{Disc}(\text{Par}(\mathbf{v}))) = 0$. If instead $\text{Par}(\sigma_p) \in A - A[\ell]$, then $v_p(\text{Disc}(\text{Par}(\mathbf{v})))$ is determined by $\mathbf{v}_Q$ as a consequence of Theorem \ref{tWildDisc}. Finally, if $\text{Par}(\sigma_p) \in A[\ell] - \{0\}$, then the proof of Theorem \ref{tWildDisc} exhibits that $v_p(\text{Disc}(\text{Par}(\mathbf{v})))$ equals $\frac{|A|}{\ell}$ multiplied by the $p$-adic valuation of the discriminant of the unique degree $\ell$ subfield of $\Q_p(\zeta_{p^\infty})$. Therefore the conductor--discriminant formula demonstrates the validity of
\[
v_p(\text{Disc}(\text{Par}(\mathbf{v}))) =
\begin{cases}
|A| \cdot \left(1 - \frac{1}{\ell}\right) & \text{if } \ell \neq p \\
2|A| \cdot \left(1 - \frac{1}{\ell}\right) & \text{if } \ell = p.
\end{cases}
\]
Therefore there exists $C_4 > 0$, depending only on $c_a$ and $\mathbf{v}_Q = (v_a)_{a \in A - A[\ell]}$, such that
\[
\text{Disc}(\text{Par}(\mathbf{v})) = C_4 \cdot \left(\prod_{a \in A[\ell] - \{0\}} \Delta(|w_a|)\right)^{\frac{|A| \cdot (\ell - 1)}{\ell}}.
\]
We are now ready to apply Theorem \ref{tMultiReduction} to the innermost sum. We shall do so with every choice of $S \subseteq A[\ell] - \{0\}$ satisfying $T + S = A[\ell]$, which allows us to detect the condition (\ref{eSurjCondition}) using inclusion--exclusion. Here we make essential use of the fact that the leading constant $C_{\text{lead}}$ from Theorem \ref{tMultiReduction} is independent of $S$.

\subsection*{The exponent of the logarithm in the main term}
Let us compute $\alpha(A)$, the exponent of the logarithm appearing in Theorem \ref{thm:main}. When we apply Theorem \ref{tMultiReduction}, the exponent of the logarithm, denoted by $\alpha$ in the theorem statement, is equal to
\[
\sum_{a \in B - \{0\}} \frac{|\textup{Lift}(H_a)|}{\varphi(\textup{lcm}(M, \ell))} + \sum_{a \in A - B} \frac{|B|}{\ell^{n - 1}} \cdot \frac{|\textup{Lift}(H_a)|}{\varphi(\textup{lcm}(M, \ell))}
\]
with $n$ the dimension of $A[\ell]$ as a $\mathbb{F}_{\ell}$-vector space. Since $\ell$ is the smallest prime divisor of $A$, it follows that $\Q(\zeta_{\ell})$ and $K$ are disjoint extensions of $\Q$. Therefore, we have
\[
\frac{|\textup{Lift}(H_a)|}{\varphi(\textup{lcm}(M, \ell))} = \frac{|H_a|}{(\ell - 1) \varphi(M)}.
\]
By construction, we have that
\[
\frac{|H_a|}{\varphi(M)} = \frac{|\widetilde{\pi}(\text{Adm}(a))| \cdot |G|}{\varphi(M)} = \frac{|\widetilde{\pi}(\text{Adm}(a))|}{[K : \Q]} = \frac{|\widetilde{\pi}(\text{Adm}(a))|}{|A/A[\ell]|}.
\]
Finally, it follows from Theorem \ref{tLocalConditions} that
\[
\frac{|\widetilde{\pi}(\text{Adm}(a))|}{|A/A[\ell]|}
=
\begin{cases}
\frac{|\text{Adm}(a)|}{|A|} & \text{if } a \in B - \{0\} \\
\frac{|\text{Adm}(a)|}{|A|} \cdot \frac{\ell^{n - 1}}{|B|} & \text{if } a \in A - B.
\end{cases}
\]
This shows that $\alpha(A)$ is the correct exponent.

\subsection*{The leading constant}
To finish the proof, we now apply Theorem \ref{tMultiReduction} for each tuple $\mathbf{v}_Q = (v_a)_{a \in A - A[\ell]}$ such that the associated map $\varphi: G_\Q \rightarrow A/A[\ell]$ satisfies (P1), (P2) and (P3). Every such tuple $\mathbf{v}_Q$ together with a choice of $\mathbf{c} = (c_a)_{a \in A[\ell] - \{0\}}$ gives a leading constant that we denote by $C_{\text{lead}}(\mathbf{v}_Q, \mathbf{c})$ and also a constant $C_4(\mathbf{v}_Q, \mathbf{c}) > 0$ satisfying
\[
\text{Disc}(\text{Par}(\mathbf{v})) = C_4(\mathbf{v}_Q, \mathbf{c}) \cdot \left(\prod_{a \in A[\ell] - \{0\}} \Delta(|w_a|)\right)^{\frac{|A| \cdot (\ell - 1)}{\ell}}.
\]
Then we take
\[
C_{\text{weak}} = \sum_{\mathbf{v}_Q} \sum_\mathbf{c} \frac{C_{\text{lead}}(\mathbf{v}_Q, \mathbf{c})}{C_4(\mathbf{v}_Q, \mathbf{c})}.
\]
Because the exponent in the discriminant for all variables outside $A[\ell]$ is bigger than the exponent for the variables in $A[\ell]$, one sees that the sum
\[
\sum_{\mathbf{v}_Q} \sum_\mathbf{c} \frac{1}{C_4(\mathbf{v}_Q, \mathbf{c})}
\]
converges and also may be truncated (to those $\mathbf{v}_Q$ satisfying that $M \leq (\log X)^{C_5}$) with an acceptable error term. Since the leading constant $C_{\text{lead}}(\mathbf{v}_Q, \mathbf{c})$ is uniformly bounded, we conclude that the sum defining $C_{\text{weak}}$ may also be truncated. Hence, the contribution from the error term in Theorem \ref{tMultiReduction} is negligible.

In order to show that $C_{\text{weak}} > 0$, it suffices to show that $C_{\text{lead}}(\mathbf{v}_Q, \mathbf{c}) > 0$ for some choice of $\mathbf{v}_Q$ and $\mathbf{c}$. We choose a splitting of $A[\ell]$ as $A[\ell] = B \oplus B_{\text{comp}}$ and a splitting $A = \widetilde{B} \oplus B_{\text{comp}}$ with $B \subseteq \widetilde{B}$. By \cite[Proposition 5.5]{FLN}, we may find a surjective homomorphism $\varphi': G_\Q \rightarrow \widetilde{B}$ such that all decomposition groups are cyclic. In particular, $\varphi'$ satisfies weak approximation. We extend $\varphi'$ to a homomorphism $\varphi'': G_\Q \rightarrow A$ by sending $g$ to $(\varphi'(g), 0)$, where we have implicitly used our splitting $A = \widetilde{B} \oplus B_{\text{comp}}$.

Now we apply Theorem \ref{tMultiReduction} with $M$ equal to the absolute discriminant of $\varphi''$, $B = A[\ell] \cap \ell A$ as above, $f(p, S) = B$ for $S$ intersecting $B$ trivially, $g(p) = \{0\}$ and $c_a = 1$. Now twisting $\varphi''$ with such multicyclic extensions gives a new $A$-extension satisfying weak approximation. This forces $C_{\text{lead}}(\mathbf{v}_Q, \mathbf{c}) > 0$ and therefore $C_{\text{weak}} > 0$.
\end{proof}

\section{The character sum}
\label{sec:charsum}
The following three sections are devoted to the proof of Theorem \ref{tMultiReduction}. Let $A = \mathbb{F}_\ell^n$ and let $\mathbf{v} = (v_a)_{a \in \mathbb{F}_\ell^n - \{0\}} \in \mathcal{A}$. Write $\pi_i$ for the projection map $\pi_i: \mathbb{F}_\ell^n \rightarrow \mathbb{F}_{\ell}$ on the $i$-th coordinate. We begin by expressing the indicator function for a tuple $\mathbf{v}$ being $f$-correct as an explicit character sum. Define
\[
\psi_i = \sum_{\substack{a \in \mathbb{F}_\ell^n \\ \pi_i(a) \neq 0}} \pi_i(a) \cdot \psi_{v_a, \ell, 1}
\]
for $i \in [n] := \{1, \dots, n\}$. We write $\psi: G_\Q \rightarrow A$ for the homomorphism given by
\[
\psi(\sigma) = (\psi_1(\sigma), \dots, \psi_n(\sigma)),
\]
so $\psi = \text{Par}(\mathbf{v})$ by construction. Our aim is to express the sum
\begin{align}
\label{eCharSumCorrect}
\sum_{\substack{\mathbf{v} = (v_a)_{a \in A - \{0\}} \in \mathcal{A} \\ \prod_{a \in \mathbb{F}_\ell^n - \{0\}} \Delta(|v_a|) \leq X \\ v_a \neq 1 \ \forall a \in S \\ v_a \equiv c_a \bmod d(\ell) \\ \gcd(v_a, M) = 1 \\ p \mid v_a \Rightarrow p \bmod M \in H_a}} \mathbf{1}_{\textup{Par}(\mathbf{v}) \ f\textup{-correct}}
\end{align}
as a sum of Dirichlet characters. First we will see how to rewrite the indicator function. We have a perfect bilinear pairing
\[
A \times A^\vee \rightarrow \mathbb{C}^\ast, \quad (a, \chi) \mapsto \chi(a).
\]
Now take some $a \in A$. Then we get an induced perfect bilinear pairing
\[
\frac{A}{\langle a \rangle} \times \{\chi \in A^\vee : \chi(a) = 1\} \rightarrow \mathbb{C}^\ast.
\]
By definition of $f$-correct, we need to check 
\begin{enumerate}
\item[(Q1)] for all $a \not \in B$ and all $p \mid v_a$ coprime to $2M$ that
\begin{align}
\label{efCorrect}
\psi(\text{Frob}_p) - \epsilon_{a, p \bmod M, f} \in B + \langle a \rangle,
\end{align}
where $\epsilon_{a, p \bmod M, f}$ is an element of $A$ depending only on $a$, $p \bmod M$ and $f$;
\item[(Q2)] let $p \mid M$. Suppose that $p \neq 2$ or $\ell \neq 2$. Then we have that $p$ is unramified and furthermore
\begin{align}
\label{egCorrect}
\psi(\text{Frob}_p) \in g(p).
\end{align}
\end{enumerate}
Write $e_i$ for the $i$-th standard basis vector. Also denote by $\chi_i: \mathbb{F}_\ell^n \rightarrow \langle \zeta_{\ell} \rangle$ the element in $A^\vee$ satisfying
\[
\chi_i(e_j) = \zeta_{\ell}^{\delta(i, j)}.
\]
Write $B_a = B + \langle a \rangle$. Then we have
\[
\{\chi \in A^\vee : \chi(B_a) = 1\} = \left\{\prod_{j \in [n]} \chi_j^{\pi_j(\mathbf{x})} : \mathbf{x} \in \mathbb{F}_\ell^n, \langle \gamma, \mathbf{x} \rangle = 0 \text{ for all } \gamma \in B_a\right\}.
\]
Using orthogonality of characters on the abelian group $\mathbb{F}_\ell^n$, we are now in the position to detect condition (Q1), i.e. equation (\ref{efCorrect}), by
\[
\frac{1}{\ell^m} \sum_{\substack{\mathbf{x} \in \mathbb{F}_\ell^n \\ \langle \gamma, \mathbf{x} \rangle = 0 \ \forall \gamma \in B_a}} \prod_{j \in [n]} \chi_j(\psi(\text{Frob}_p))^{\pi_j(\mathbf{x})} \prod_{j \in [n]} \overline{\chi_j(\epsilon_{a, p \bmod M, f})}^{\pi_j(\mathbf{x})},
\]
where $m$ equals $n - 1 - \dim_{\FF_{\ell}} B$. We observe that $\chi_j \circ \psi$ is simply $\psi_j$ after identifying $\mathbb{F}_{\ell}$ with $\langle \zeta_{\ell} \rangle$ by sending $1$ to $\zeta_\ell$. Therefore we can rewrite the above as
\[
\frac{1}{\ell^m} \sum_{\substack{\mathbf{x} \in \mathbb{F}_\ell^n \\ \langle \gamma, \mathbf{x} \rangle = 0 \ \forall \gamma \in B_a}} \prod_{j \in [n]} \psi_j(\text{Frob}_p)^{\pi_j(\mathbf{x})} \delta_{a, p \bmod M, j, f}^{\pi_j(\mathbf{x})},
\]
where we have renamed $\overline{\chi_j(\epsilon_{a, p \bmod M, f})}$ as $\delta_{a, p \bmod M, j, f}$. We may similarly detect condition (Q2), i.e. equation (\ref{egCorrect}), by the simultaneous conditions $v_a \not \equiv 0 \bmod p$ for all $a$ and
\begin{align}
\label{eC2detect}
\frac{1}{\ell^n} \sum_{\alpha \in g(p)} \sum_{\mathbf{x} \in \mathbb{F}_\ell^n} \prod_{j \in [n]} \psi_j(\text{Frob}_p)^{\pi_j(\mathbf{x})} \overline{\chi_j(\alpha)}^{\pi_j(\mathbf{x})}.
\end{align}
We will now return the character sum for only the condition (Q1). It will be straightforward to then insert the condition (Q2) later. With this in hand, we are now able to rewrite the indicator function for the condition (Q1) as
\[
\mathbf{1}_{f\text{ satisfies (Q1)}} = 
\ell^{-m\sum_{a \in \mathbb{F}_\ell^n - B} \tilde{\omega}(v_a)} \prod_{a \in \mathbb{F}_\ell^n - B} \prod_{\substack{p \mid v_a \\ p \equiv 1 \bmod 2}} \sum_{\substack{\mathbf{x} \in \mathbb{F}_\ell^n \\ \langle \gamma, \mathbf{x} \rangle = 0 \ \forall \gamma \in B_a}} \prod_{j \in [n]} \psi_j^{\pi_j(\mathbf{x})}(\text{Frob}_p) \delta_{a, p \bmod M, j, f}^{\pi_j(\mathbf{x})},
\]
where $\tilde{\omega}$ denotes the number of odd prime divisors. The product equals
\[
\prod_{a \in \mathbb{F}_\ell^n - B} \sum_{\substack{(\mathbf{x}_p)_{p \mid v_a \text{ odd}} \\ \mathbf{x}_p \in \mathbb{F}_\ell^n \\ \langle \gamma, \mathbf{x}_p \rangle = 0 \ \forall \gamma \in B_a}} \prod_{\substack{p \mid v_a \\ p \equiv 1 \bmod 2}} \prod_{j \in [n]} \psi_j^{\pi_j(\mathbf{x}_p)}(\text{Frob}_p) \delta_{a, p \bmod M, j, f}^{\pi_j(\mathbf{x}_p)},
\]
which is in turn
\[
\sum_{\substack{(\mathbf{x}_{p, a})_{p \mid v_a \text{ odd}, a \in \mathbb{F}_\ell^n - B} \\ \mathbf{x}_{p, a} \in \mathbb{F}_\ell^n \\ \langle \gamma, \mathbf{x}_{p, a} \rangle = 0 \ \forall \gamma \in B_a}} \prod_{a \in \mathbb{F}_\ell^n - B} \prod_{\substack{p \mid v_a \\ p \equiv 1 \bmod 2}} \prod_{j \in [n]} \psi_j^{\pi_j(\mathbf{x}_{p, a})}(\text{Frob}_p) \delta_{a, p \bmod M, j, f}^{\pi_j(\mathbf{x}_{p, a})}.
\]
For every $a \in \mathbb{F}_\ell^n - B$ and every $b \in \mathbb{F}_\ell^n$ with $\langle \gamma, b \rangle = 0$ for all $\gamma \in B_a$, we introduce a new variable
\[
w_{a, b} = \prod_{\substack{p \mid v_a \\ p \equiv 1 \bmod 2 \\ \mathbf{x}_{p, a} = b}} p, \quad w_{a, \bullet} = \text{sgn}(v_a) \cdot 2^{v_2(v_a)}.
\]
We can recover $v_a$ and the vectors $\mathbf{x}_{p, a}$ from the $w_{a, b}$. Indeed, we have
\[
v_a = w_{a, \bullet} \prod_{\substack{b \in \mathbb{F}_\ell^n \\ \langle \gamma, b \rangle = 0 \ \forall \gamma \in B_a}} w_{a, b}.
\]
Now to find $\mathbf{x}_{p, a}$, note that $p \mid v_a$ by the definition of $\mathbf{x}_{p, a}$. Therefore we may take the unique $p$ such that $p \mid w_{a, b}$. Then we reconstruct $\mathbf{x}_{p, a}$ by taking this $b$. This transforms the sum into
\[
\sum_{(v_a)_{a \in B - \{0\}}} \sump_{\substack{(w_{a, b})_{a, b} \\ \prod_a \Delta(|v_a|) \prod_{a, b} \Delta(|w_{a, b}|) \leq X}} \hspace{-0.7cm} \ell^{-m \sum_{a, b} \tilde{\omega}(w_{a, b})} \prod_{a \in \mathbb{F}_\ell^n - B} \prod_{\substack{p \mid v_a \\ p \equiv 1 \bmod 2}} \prod_{j \in [n]} \psi_j^{\pi_j(b)}(\text{Frob}_p) \delta_{a, p \bmod M, j, f}^{\pi_j(b)},
\]
where $\sump$ also includes the following additional summation conditions
\[
v_a \equiv c_a \bmod d(\ell), \quad \gcd(v_a, M) = 1, \quad p \mid w_{a, b} \Rightarrow p \bmod M \in H_a
\]
and
\[
v_a \neq 1 \ \forall a \in S, \quad p \mid v_a \Rightarrow p \bmod M \in H_a, \quad \mathbf{v} \in \mathcal{A}.
\]
We now expand $\psi_j$ to deduce that the above sum equals
\[
\sum_{(v_a)_{a \in B - \{0\}}} \sump_{\substack{(w_{a, b})_{a, b} \\ \prod_a \Delta(|v_a|) \prod_{a, b} \Delta(|w_{a, b}|) \leq X}} \hspace{-0.7cm} \ell^{-m \sum_{a, b} \tilde{\omega}(w_{a, b})} \prod_{a_1, b_1} \epsilon' \prod_{p \mid w_{a_1, b_1}} \prod_{j \in [n]} \prod_{\substack{a_2 \in \mathbb{F}_\ell^n \\ \pi_j(a_2) \neq 0}} \psi_{v_{a_2}, \ell, 1}^{\pi_j(b_1) \pi_j(a_2)}(\text{Frob}_p),
\]
where 
\[
\epsilon' = \prod_{j \in [n]} \prod_{p \mid w_{a_1, b_1}} \delta_{a_1, p \bmod M, j, f}^{\pi_j(b_1)}.
\]
By construction we have $\langle b_1, a_2 \rangle = 0$ for $a_2 \in B$. Therefore we may further expand $\psi_j$ to rewrite the above sum as
\[
\sum_{(v_a)_{a \in B - \{0\}}} \sump_{\substack{(w_{a, b})_{a, b} \\ \prod_a \Delta(|v_a|) \prod_{a, b} \Delta(|w_{a, b}|) \leq X}} \hspace{-0.75cm} \ell^{-m \sum_{a, b} \tilde{\omega}(w_{a, b})} \prod_{a_1, b_1} \epsilon' \prod_{p \mid w_{a_1, b_1}} \prod_{j \in [n]} \prod_{\substack{a_2 \in \mathbb{F}_\ell^n \\ \pi_j(a_2) \neq 0}} \prod_{b_2} \psi_{w_{a_2, b_2}, \ell, 1}^{\pi_j(b_1) \pi_j(a_2)}(\text{Frob}_p),
\]
where the product over $b_2$ includes $\bullet$, while the product over $b_1$ does not. Using the definition of $\langle \cdot, \cdot \rangle$, we may finally rewrite this as
\[
\sum_{(v_a)_{a \in B - \{0\}}} \sump_{\substack{(w_{a, b})_{a, b} \\ \prod_a \Delta(|v_a|) \prod_{a, b} \Delta(|w_{a, b}|) \leq X}} \ell^{-m \sum_{a, b} \tilde{\omega}(w_{a, b})} \prod_{a_1, b_1} \epsilon' \prod_{a_2, b_2} \psi_{w_{a_2, b_2}, \ell, 1}^{\langle b_1, a_2 \rangle}(\text{Frob}_{w_{a_1, b_1}}).
\]
Inserting the conditions from (Q2), see equation (\ref{eC2detect}), and writing out the implicit summation conditions in $\sum'$, we conclude that
\[
\sum_{\substack{\mathbf{v} = (v_a)_{a \in A - \{0\}} \in \mathcal{A} \\ \prod_{a \in \mathbb{F}_\ell^n - \{0\}} \Delta(|v_a|) \leq X \\ v_a \neq 1 \ \forall a \in S \\ v_a \equiv c_a \bmod d(\ell) \\ \gcd(v_a, M) = 1 \\ p \mid v_a \Rightarrow p \bmod M \in H_a}} \mathbf{1}_{\textup{Par}(\mathbf{v}) \ f\textup{-correct}}
\]
equals
\begin{multline}
\label{eNX}
N(X) = \sum_{\substack{(v_a)_{a \in B - \{0\}} \\ p \mid v_a \Rightarrow p \bmod M \in H_a}} \sum_{\substack{(w_{a, b})_{a, b} \\ \prod_a \Delta(|v_a|) \prod_{a, b} \Delta(|w_{a, b}|) \leq X \\ v_a \equiv c_a \bmod d(\ell) \\ \gcd(v_a, M) = 1 \\ p \mid w_{a, b} \Rightarrow p \bmod M \in H_a}} \mathbf{1}_{v_a \neq 1 \ \forall a \in S} \times \mathbf{1}_{\mathbf{v} \in \mathcal{A}} \times \\
\prod_{p \mid M'} \left(\frac{1}{\ell^n} \sum_{\alpha \in g(p)} \sum_{\mathbf{x} \in \mathbb{F}_\ell^n} \prod_{j \in [n]} \psi_j(\text{Frob}_p)^{\pi_j(\mathbf{x})} \overline{\chi_j(\alpha)}^{\pi_j(\mathbf{x})}\right) \times \\
\ell^{-m \sum_{a, b} \tilde{\omega}(w_{a, b})} \prod_{a_1, b_1} \epsilon' \prod_{a_2, b_2} \psi_{w_{a_2, b_2}, \ell, 1}^{\langle b_1, a_2 \rangle}(\text{Frob}_{w_{a_1, b_1}}),
\end{multline}
where
\[
M' = 
\begin{cases}
\frac{M}{2^{v_2(M)}} & \text{if } \ell = 2 \\
M & \text{if } \ell > 2.
\end{cases}
\]
We will now see how to find the main term of the above sum.

\section{Combinatorial considerations and the main term}
\label{sComb}
When the characters appearing in the sum described above are non-trivial then they will oscillate, giving rise to cancellation in the sum. However, should the combination of characters cancel then the contribution to the sum will be much larger. The purpose of this section is to develop combinatorial conditions on the indices $a,b$ such that the resulting combination of characters yields a principal character and hence a dominant contribution to the sum. The reader should compare this process to that carried out in \cite{biquad}, or in \cite{FK}, in order to identify the main term of their sums of combinations of Legendre symbols. 

\subsection{Multiquadratic case}
Define
\[
\mathcal{I} := \{(S, T) : S, T \in \mathbb{F}_2^n, S \not \in B, \langle \gamma, T \rangle = 0 \ \forall \gamma \in B_S\}.
\]
For a subspace $V$ of $\mathbb{F}_2^n$ we write
\[
V^\top = \{S \in \mathbb{F}_2^n : \langle S, v \rangle = 0 \ \forall v \in V\}
\]
for the complement under the pairing $\langle \cdot, \cdot \rangle$. We have the following crucial lemma.

\begin{lemma}
\label{lBlock}
Let $X \subseteq \mathcal{I}$ with $|X| \geq 2^n - |B|$. Suppose that
\begin{align}
\label{eConstantParity}
\langle S_1, T_2 \rangle + \langle S_2, T_1 \rangle = 0
\end{align}
for all $(S_1, T_1), (S_2, T_2) \in X$. Then we have
\[
X = \{(S, f(S)) : S \in \mathbb{F}_2^n - B\}
\]
for some function $f: \mathbb{F}_2^n - B \rightarrow B^\top$ which is alternating with respect to the bilinear pairing $\langle \cdot, \cdot \rangle$, in the sense that
\[
\langle S, f(T) \rangle = \langle T, f(S) \rangle \textup{ and } \langle S, f(S) \rangle = 0.
\]
\end{lemma}

\begin{proof}
Denote by $\pi_1$ and $\pi_2$ the natural projection maps from $\mathcal{I}$ to $\mathbb{F}_2^n$. Write $V_1$ for the subspace generated by $\pi_1(X)$ and write $V_2$ for the subspace generated by $\pi_2(X)$. By construction of $\mathcal{I}$ we have that
\begin{align}
\label{eV2}
V_2 \subseteq B^\top.
\end{align}
By the pigeonhole principle there exists some $T_0 \in \pi_2(X)$ such that 
\begin{align}
\label{eT0fiber}
|\pi_2^{-1}(T_0) \cap X| \geq \frac{2^n - |B|}{|V_2|}.
\end{align}
List the elements of $\pi_2^{-1}(T_0) \cap X$ as
\[
(S_1, T_0), \dots, (S_\alpha, T_0).
\]
Suppose that there exists $T_1 \in \pi_2(X)$ such that
\[
\langle S_i, T_1 \rangle + \langle S_j, T_1 \rangle = 1
\]
for some $1 \leq i, j \leq \alpha$. We claim that this contradicts equation (\ref{eConstantParity}). Indeed, take such a $T_1$ and take such $i, j$. Let $U$ be such that $(U, T_1) \in X$. Then either $(U, T_1)$ and $(S_i, T_0)$ contradict equation (\ref{eConstantParity}) or $(U, T_1)$ and $(S_j, T_0)$ do.

Therefore we may assume from now on that that for all $T_1 \in \pi_2(X)$ and all $1 \leq i, j \leq \alpha$
\[
\langle S_i, T_1 \rangle + \langle S_j, T_1 \rangle = 0.
\]
We conclude that $S_i - S_j \in V_2^\top$. We now claim that $V_1$ contains $V_2^\top$, so it contains in particular $B$ by equation (\ref{eV2}). Since we have $S_i - S_j \in V_2^\top$, we now consider the elements
\[
\{S_j - S_1 : 2 \leq j \leq \alpha\}.
\]
This gives $\alpha - 1 = |\pi_2^{-1}(T_0) \cap X| - 1$ non-zero elements of $V_2^\top \cap V_1$. Therefore we get from equation (\ref{eT0fiber}) that
\[
|V_2^\top \cap V_1| \geq \alpha \geq \frac{2^n - |B|}{|V_2|} = |V_2^\top| - \frac{|B|}{|V_2|}
\]
after also taking into account the zero element of $V_2^\top \cap V_1$. If $B$ has codimension $0$ or $1$, then the lemma is trivial. Otherwise we have
\[
|V_2^\top| - \frac{|B|}{|V_2|} \geq |V_2^\top| - \frac{|V_2^\top|}{4}.
\]
We conclude that
\[
|V_2^\top \cap V_1| \geq \frac{3|V_2^\top|}{4},
\]
which readily implies the claim. Thanks to the claim we see that $V_1$ contains $B$.

We next claim that $V_1$ equals $\mathbb{F}_2^n$. We now fix some $S_0 \in \pi_1(X)$. Arguing as before, we see that
\[
\langle S_1, T_i \rangle + \langle S_1, T_j \rangle = 0
\]
for all $S_1 \in \pi_1(X)$ and all $i$ and $j$ such that $(S_0, T_i), (S_0, T_j) \in X$. In particular, we deduce that $|\pi_1^{-1}(S_0) \cap X| \leq |V_1^\top|$. We now sum to obtain
\[
2^n - |B| \leq |X| = \sum_{S_0 \in \pi_1(X)} |\pi_1^{-1}(S_0) \cap X| \leq \sum_{S_0 \in \pi_1(X)} |V_1^\top| \leq (|V_1| - |B|) \cdot |V_1^\top| = 2^n - |B| \cdot |V_1^\top|,
\]
because $V_1$ contains $B$. But this is only possible if $|V_1^\top| = 1$ or equivalently $V_1 = \mathbb{F}_2^n$. Therefore there exists a function $f: \mathbb{F}_2^n - B \rightarrow B^\top$ such that
\[
X = \{(S, f(S)) : S \in \mathbb{F}_2^n - B\}.
\]
It is readily verified that $f$ must be alternating, completing the proof.
\end{proof}

\subsection{Multicyclic case}
Define
\[
\mathcal{I} := \{(a, b) : a, b \in \mathbb{F}_\ell^n, a \not \in B, \langle \gamma, b \rangle = 0 \ \forall \gamma \in B_a\},
\]
where we recall that $\langle \cdot, \cdot \rangle$ is the standard bilinear form. We have the following crucial lemma.

\begin{lemma}
\label{lBlock2}
Let $X \subseteq \mathcal{I}$ with $|X| \geq \ell^n - |B|$. Suppose that
\begin{align}
\label{eConstantPairingOdd}
\langle a_2, b_1 \rangle = 0
\end{align}
for all $(a_1, b_1), (a_2, b_2) \in X$. Then we have
\[
X = \{(a, 0) : a \in \mathbb{F}_\ell^n - B\}.
\]
\end{lemma}

\begin{proof}
The proof will be similar to the proof of Lemma \ref{lBlock}. We write $\pi_1$ and $\pi_2$ for the natural projection maps from $\mathcal{I}$ to $\mathbb{F}_\ell^n$. We denote by $N$ the number of elements in the subspace generated by $\pi_2(X)$. If $N = 1$, then we have
\[
X = \{(a, 0) : a \in \mathbb{F}_\ell^n - B\}.
\]
From now on we may and will assume that $N > 1$, and seek a contradiction.

Take some $b_0 \in \pi_2(X)$ and suppose that $|\pi_2^{-1}(b_0) \cap X| \geq (\ell^n - |B|)/N$. The existence of such a $b_0$ is guaranteed by the pigeonhole principle and our assumption $|X| \geq \ell^n - |B|$. We enumerate the elements of $\pi_2^{-1}(b_0) \cap X$ as
\[
(a_1, b_0), \dots, (a_k, b_0)
\]
with $k \geq (\ell^n - |B|)/N$. Then we have, thanks to equation (\ref{eConstantPairingOdd}), the equality
\[
\langle a_i, b \rangle = 0
\]
for all $b \in \pi_2(X)$ and all $1 \leq i \leq k$. Since $\langle \cdot, \cdot \rangle$ is non-degenerate and $N$ is the cardinality of the subspace generated by $\pi_2(X)$, it follows that there exists a subspace $V$ of dimension $n - \log_\ell N$ containing $B$ such that $a_i \in V$ for all $1 \leq i \leq k$. Furthermore, we know that the $a_i$ are not in $B$. This gives the inclusion
\[
\{a_i : 1 \leq i \leq k\} \subseteq V \setminus B
\]
and therefore the bound $k \leq \frac{\ell^n}{N} - |B|$. Therefore we conclude that
\[
\frac{\ell^n}{N} - |B| \geq k = |\pi_2^{-1}(b_0) \cap X| \geq \frac{\ell^n - |B|}{N},
\]
which is a contradiction for $N > 1$.
\end{proof}

\section{Oscillation of characters}\label{sec:oscillation}
We now return to equation (\ref{eNX}). We say that an integer $x$ is large if
\[
|x| > \exp\left((\log X)^{A_1}\right),
\]
where $A_1 > 0$ is a small constant that we will choose later. 

\subsection{Large variables}
\label{ssLarge}
We split the character sum $N(X)$ in two subsums
\[
N(X) = N_{\text{small}}(X) + N_{\text{large}}(X),
\]
where $N_{\text{small}}(X)$ is by definition the contribution to $N(X)$, where at most $\ell^n - |B| - 1$ of the variables $w_{a, b}$ are large, and $N_{\text{large}}(X)$ is by definition the remaining contribution. We will make use of the following well-known lemma.

\begin{lemma}
\label{lBound}
Let $\kappa, C > 0$ be fixed real numbers. Then we have the bounds
\[
\sum_{\substack{1 \leq n \leq x \\ p \mid n \Rightarrow p \bmod M \in H}} \mu^2(n) \kappa^{\omega(n)} \ll_{\kappa, C} x (\log x)^{\frac{|H| \cdot \kappa}{\varphi(M)} - 1}
\]
and
\[
\sum_{\substack{1 \leq n \leq x \\ p \mid n \Rightarrow p \bmod M \in H}} \frac{\mu^2(n) \kappa^{\omega(n)}}{n} \ll_{\kappa, C} (\log x)^{\frac{|H| \cdot \kappa}{\varphi(M)}}
\]
for all $M \leq (\log x)^C$ and all subsets $H$ of $(\Z/M\Z)^\ast$.
\end{lemma}

\begin{proof}
The first bound follows from Theorem \ref{thm:GK} with 
$$
f(n) := \mu^2(n) \kappa^{\omega(n)} \mathbf{1}_{p \mid n \Rightarrow p \bmod M \in H}, \quad \alpha := \frac{|H| \cdot \kappa}{\varphi(M)}, \quad k := \kappa, \quad Q := \exp\left((\log x)^{\min\left(\frac{1}{4\kappa}, \frac{1}{2}\right)}\right).
$$
and $J := 1$, $\epsilon := 1/10$. The assumption (\ref{eOnPrimes}) is guaranteed by the Siegel--Walfisz theorem. The second bound follows from partial summation and the first bound.
\end{proof}

After choosing the constant $A_1 > 0$ to be sufficiently small, it follows from Lemma \ref{lBound} that we have the bound
\[
N_{\text{small}}(X) = O\left(X (\log X)^{\alpha - 1 - \delta}\right)
\]
for some $\delta > 0$. It is precisely at this step that we make fundamental use of the assumptions $|H_a| = |H_{a'}|$ for all $a, a' \in A - B$ and $[(\Z/M\Z)^\ast : H_a] \leq C$.

\subsection{Linked variables}
\label{ssLinked}
We will now turn our attention to $N_{\text{large}}(X)$. We say that an integer $x$ is medium if
\[
|x| > (\log X)^{A_2},
\]
where $A_2 > 0$ is a large constant to be chosen later. We also split the sum $N_{\text{large}}(X)$ in two subsums, namely
\[
N_{\text{large}}(X) = N_{\text{linked}}(X) + N_{\text{main}}(X).
\]
Here $N_{\text{linked}}(X)$ is the contribution to $N_{\text{large}}(X)$ for which the following holds
\begin{itemize}
\item if $\ell > 2$, then there exists $(a_1, b_1), (a_2, b_2) \in \mathcal{I}$ such that all of the following conditions hold
\begin{enumerate}
\item[(1)] we have $\langle a_2, b_1 \rangle \neq 0$ or $\langle a_1, b_2 \rangle \neq 0$;
\item[(2)] we have that $w_{a_1, b_1}$ and $w_{a_2, b_2}$ are both medium or we have that $|w_{a_1, b_1}|, |w_{a_2, b_2}| > 1$ and one of the $w_{a_i, b_i}$ is large.
\end{enumerate}
\item if $\ell = 2$, then there exists $(a_1, b_1), (a_2, b_2) \in \mathcal{I}$ such that all of the following conditions hold
\begin{enumerate}
\item[(1)] we have $\langle a_2, b_1 \rangle + \langle a_1, b_2 \rangle = 1$;
\item[(2)] we have that $w_{a_1, b_1}$ and $w_{a_2, b_2}$ are both medium or we have that $|w_{a_1, b_1}|, |w_{a_2, b_2}| > 2$ and one of the $w_{a_i, b_i}$ is large.
\end{enumerate}
\end{itemize}
Furthermore, $N_{\text{main}}(X)$ is by definition the remaining contribution.

Our next goal is to bound $N_{\text{linked}}(X)$. Our two principal tools are the large sieve, as presented in Proposition \ref{pLargeSieve}, and the Siegel--Walfisz theorem over number fields as presented in the main theorem of \cite{Goldstein}.

\subsubsection{Equidistribution with the large sieve}
\label{ssLS}
We will now bound $N_{\text{linked}}(X)$. We will first suppose that there exist $(a_1, b_1), (a_2, b_2) \in \mathcal{I}$ with $w_{a_1, b_1}$ and $w_{a_2, b_2}$ both medium and satisfying the aforementioned conditions. So fix such a choice of $(a_1, b_1)$ and $(a_2, b_2)$. Define $\mathcal{I}' := \mathcal{I} - \{(a_1, b_1), (a_2, b_2)\}$. The corresponding contribution to $N_{\text{linked}}(X)$ is bounded by
\begin{multline}
\label{eNlX}
\sum_{(v_a)_{a \in B - \{0\}}} \sum_{(w_{a, \bullet})} \sum_{(w_{a, b})_{(a, b) \in \mathcal{I'}}} \ell^{-m \sum_{(a, b) \in \mathcal{I}'} \tilde{\omega}(w_{a, b})}\\
\left|\sum_{\substack{w_{a_1, b_1}, w_{a_2, b_2} \\ \Delta(|w_{a_1, b_1} w_{a_2, b_2}|) \leq \frac{X}{\prod_a \Delta(|v_a|) \prod_{(a, b) \in \mathcal{I}'} \Delta(|w_{a, b}|)}}} \hspace{-2cm} \alpha_{w_{a_1, b_1}} \beta_{w_{a_2, b_2}} \psi_{w_{a_2, b_2}, \ell, 1}^{\langle b_1, a_2 \rangle}(\text{Frob}_{w_{a_1, b_1}}) \psi_{w_{a_1, b_1}, \ell, 1}^{\langle b_2, a_1 \rangle}(\text{Frob}_{w_{a_2, b_2}})\right|,
\end{multline}
where $\alpha_{w_{a_1, b_1}}$ and $\beta_{w_{a_2, b_2}}$ are complex numbers of absolute value bounded by $1$ depending only on respectively $w_{a_1, b_1}$ and $w_{a_2, b_2}$ (and $w_{a, b}$ for $(a, b) \in \mathcal{I}'$). By changing the coefficients if necessary, we may assume from now on that $w_{a_1, b_1}$ and $w_{a_2, b_2}$ are coprime to $\ell$.

We now work towards our goal of applying Proposition \ref{pLargeSieve}. We take $K = \Q(\zeta_{\ell})$ and take the $M$ of Proposition \ref{pLargeSieve} to be a sufficiently large power of $\ell$. Critically, the field $K$ depends only on the abelian group $A$. Write $(\cdot/\cdot)_{\Q(\zeta_{\ell}), \ell}$ for the $\ell$-th power residue symbol in $\Q(\zeta_{\ell})$. We define
\begin{align}
\label{egamma}
\gamma(w, z) := \left(\frac{N_{\Q(\zeta_{\ell})/\Q}(w)}{z}\right)_{\Q(\zeta_{\ell}), \ell}^{\langle b_1, a_2 \rangle} \left(\frac{N_{\Q(\zeta_{\ell})/\Q}(z)}{w}\right)_{\Q(\zeta_{\ell}), \ell}^{\langle b_2, a_1 \rangle}.
\end{align}
Note that if we define $\widetilde{\gamma}(w, z) := \gamma(z, w)$, then $\widetilde{\gamma}(w, z)$ is still of the shape (\ref{egamma}). Therefore, if we check properties (P1) -- (P3) for all $\gamma(w, z)$, then these properties will also hold for $\widetilde{\gamma}(w, z)$. This allows us to circumvent the condition $X \leq Y$ in Proposition \ref{pLargeSieve} by applying Proposition \ref{pLargeSieve} to $\gamma(w, z)$ or $\widetilde{\gamma}(w, z)$ depending on whether $X$ or $Y$ is larger. Let us now verify (P1) -- (P3).

Property (P1) is clear. We take $A_{\text{bad}}$ to be the set of squarefull integers. In particular, property (P3) is immediate for $C_2 = 1/2$ if we take $C_1$ sufficiently large. The first part of property (P2) follows from reciprocity and the periodicity of power residue symbols provided that we take the $M$ from Proposition \ref{pLargeSieve} to be a sufficiently large power of $\ell$. It remains to prove the final part of property (P2).

To this end, fix some $w$. Then the application $z \mapsto \gamma(w, z)$ is a multiplicative character with period $M N_{\Q(\zeta_{\ell})/\Q}(w)$ by assumption. Therefore, by orthogonality of characters, it suffices to show that the character $z \mapsto \gamma(w, z)$ is not the principal character. By assumption, we have that $N_{\Q(\zeta_{\ell})/\Q}(w)$ is not squarefull. Therefore we may take a prime ideal $\mathfrak{p}$ of $\Q(\zeta_{\ell})$ of degree $1$ that divides $w$ such that none of the conjugates of $\mathfrak{p}$ divides $w$. By the Chinese Remainder Theorem, we may find an element $z \in \Z[\zeta_{\ell}]$ such that
\[
z \equiv 1 \bmod M, \quad z \equiv 1 \bmod \mathfrak{q} \text{ for all } \mathfrak{q} \mid N_{\Q(\zeta_{\ell})/\Q}(w) \text{ with } \mathfrak{q} \neq \mathfrak{p}, \quad z \equiv \alpha \bmod \mathfrak{p},
\]
where $\alpha$ is any generator of the cyclic group $(\Z[\zeta_{\ell}]/\mathfrak{p})^\ast \cong \FF_p^\ast$. Using our assumptions on $\langle b_1, a_2 \rangle$ and $\langle b_2, a_1 \rangle$, it is not hard to show now that $z \mapsto \gamma(w, z)$ is not the principal character. We have finished checking that $\gamma(\cdot, \cdot)$ satisfies all the conditions of Proposition \ref{pLargeSieve}. 

However, in equation (\ref{eNlX}), we are at the moment summing over rational integers and not over elements of $\Z[\zeta_{\ell}]$. Therefore we aim to replace the sum over the integers by a sum taking place in $\Q(\zeta_{\ell})$. The following lemma is critical.

\begin{lemma}
\label{lChoice}
Fix a root of unity $\zeta_\ell \in \mathbb{C}$ inducing an identification between $\mathbb{F}_\ell$ and $\langle \zeta_\ell \rangle$ by sending $1$ to $\zeta_\ell$. Let $p$ be a prime number such that $p \equiv 1 \bmod \ell$. Then there are canonical bijections
\begin{align*}
\{\phi \in \textup{Epi}(G_\Q, \mathbb{F}_\ell) : \phi \textup{ ramified only at } p\} &\leftrightarrow \{\textup{Dirichlet characters modulo } p \textup{ of order } \ell\} \\
&\leftrightarrow \{\textup{prime ideals of } \Q(\zeta_\ell) \textup{ above } p\}.
\end{align*}
\end{lemma}

\begin{proof}
Given a map $\phi \in \textup{Epi}(G_\Q, \mathbb{F}_\ell)$ that is only ramified at $p$, class field theory implies that $\phi$ factors through $\Gal(\Q(\zeta_p)/\Q)$, which is canonically isomorphic to $(\Z/p\Z)^\ast$. Therefore we may associate to $\phi$ an epimorphism from $(\Z/p\Z)^\ast$ to $\mathbb{F}_\ell$. Using our identification between $\mathbb{F}_\ell$ and $\langle \zeta_\ell \rangle$, $\phi$ induces a map $(\Z/p\Z)^\ast \rightarrow \mathbb{C}^\ast$. Extending $\phi$ in the usual way to $\Z$ gives a Dirichlet character modulo $p$ of order equal to $\ell$. This defines the first bijection.

For the second bijection, suppose that we are given an ideal $\mathfrak{p}$ of $\Q(\zeta_\ell)$ above $p$. Since $\mathfrak{p}$ has degree $1$, one readily verifies that the map
\[
n \mapsto \left(\frac{n}{\mathfrak{p}}\right)_{\Q(\zeta_\ell), \ell}
\]
is a Dirichlet character. One directly shows that this is a bijection as well, completing the proof of the lemma.
\end{proof}

By Lemma \ref{lChoice}, the choice of $\psi_{p, \ell, 1}$ from Section \ref{sec:par} is equivalent to choosing a prime ideal $\mathfrak{p}$ of $\Q(\zeta_{\ell})$ above $p$. Observe that $p$ indeed splits in $\Q(\zeta_{\ell})$, because we have the congruence $p \equiv 1 \bmod \ell$ if $\psi_{p, \ell, 1}$ is a homomorphism. We write this unique ideal $\mathfrak{p}$ as $\text{Pref}_\ell(p)$. We extend the definition of $\text{Pref}_\ell(p)$ multiplicatively to a function $\text{Pref}_\ell(x)$ for all squarefree integers $x$ supported in primes congruent to $1$ modulo $\ell$.

We fix a set of integral ideals $I_1, \dots, I_t$ representing every ideal class of $\text{Cl}(\Q(\zeta_{\ell}))$. We now split equation (\ref{eNlX}) in $t^2$ sums, where we insert the additional condition that
\[
\text{Pref}_\ell(w_{a_1, b_1}) \sim I_{s_1}, \quad \text{Pref}_\ell(w_{a_2, b_2}) \sim I_{s_2},
\]
where we fixed some integers $1 \leq s_1, s_2 \leq t$. We introduce new variables $w$ and $z$ in the fundamental domain of $K$ such that
\[
(w) I_{s_1} = \text{Pref}_\ell(w_{a_1, b_1}), \quad (z) I_{s_2} = \text{Pref}_\ell(w_{a_2, b_2}).
\]
We also define new coefficients
\[
\alpha_w
=
\begin{cases}
\alpha_{N_{\Q(\zeta_{\ell})/\Q}(w I_{s_1})} \cdot \left(\frac{N_{\Q(\zeta_{\ell})/\Q}(w)}{I_{s_2}}\right)_{\Q(\zeta_{\ell}), \ell}^{\langle b_2, a_1 \rangle} \cdot \left(\frac{N_{\Q(\zeta_{\ell})/\Q}(I_{s_2})}{w}\right)_{\Q(\zeta_{\ell}), \ell}^{\langle b_1, a_2 \rangle} & \text{if } (w) I_{s_1} \in \text{Im}(\text{Pref}_\ell) \\
0 & \text{otherwise}
\end{cases}
\]
and
\[
\beta_z
=
\begin{cases}
\beta_{N_{\Q(\zeta_{\ell})/\Q}(z I_{s_2})} \cdot \left(\frac{N_{\Q(\zeta_{\ell})/\Q}(z)}{I_{s_1}}\right)_{\Q(\zeta_{\ell}), \ell}^{\langle b_1, a_2 \rangle} \cdot \left(\frac{N_{\Q(\zeta_{\ell})/\Q}(I_{s_1})}{z}\right)_{\Q(\zeta_{\ell}), \ell}^{\langle b_2, a_1 \rangle} & \text{if } (z) I_{s_2} \in \text{Im}(\text{Pref}_\ell) \\
0 & \text{otherwise.}
\end{cases}
\]
Then the inner sum of equation (\ref{eNlX}) becomes $t^2$ sums of the shape
\begin{align}
\label{eLargeSieve}
\frac{1}{(\ell - 1)^2} \left|\sum_w \sum_z \alpha_w \beta_z \gamma(w, z) \right|,
\end{align}
where we divide by $\frac{1}{(\ell - 1)^2}$, because the fundamental domain of $\Q(\zeta_{\ell})$ contains $\ell - 1 = \varphi(\ell)$ generators for each principal ideal.

Finally, there is one more barrier to overcome before we are ready to apply Proposition \ref{pLargeSieve}. The implicit condition in equation (\ref{eLargeSieve}) is that
\[
N_{\Q(\zeta_{\ell})/\Q}(wz) \leq B
\]
for some bound $B > 0$, while Proposition \ref{pLargeSieve} applies only to box shapes given by $N_{\Q(\zeta_{\ell})/\Q}(w) \leq W$ and $N_{\Q(\zeta_{\ell})/\Q}(z) \leq Z$. To this end, we split $w$ and $z$ in intervals of the shape
\[
W \leq N_{\Q(\zeta_{\ell})/\Q}(w) \leq W \left(1 + \frac{1}{(\log X)^{A_3}}\right), \quad Z \leq N_{\Q(\zeta_{\ell})/\Q}(z) \leq Z \left(1 + \frac{1}{(\log X)^{A_3}}\right).
\]
This does not cover the entire region, but for sufficiently large $A_3$ one may bound the resulting leftover trivially. Inserting the bound of Proposition \ref{pLargeSieve} for each such sum into equation (\ref{eNlX}) and summing trivially shows that $N_{\text{linked}}(X)$ ends up in the error term, in the case where $w_{a_1, b_1}$ and $w_{a_2, b_2}$ are both medium, upon choosing $A_2$ sufficiently large in terms of $A_3$.

\subsubsection{Equidistribution with Siegel--Walfisz}
It is now time to bring the Siegel--Walfisz theorem into play. Let $(a_1, b_1) \in \mathcal{I}$ be such that $w_{a_1, b_1}$ is large. By definition of $N_{\text{linked}}(X)$, we know that there exists $(a_2, b_2) \in \mathcal{I}$ satisfying
\begin{itemize}
\item $|w_{a_2, b_2}| > 2$;
\item $\langle a_2, b_1 \rangle \neq 0$ or $\langle a_1, b_2 \rangle \neq 0$ if $\ell > 2$;
\item $\langle a_2, b_1 \rangle + \langle a_1, b_2 \rangle = 1$ if $\ell = 2$.
\end{itemize}
Furthermore, by the work done in Section \ref{ssLS}, we may assume that all pairs $(a_2, b_2)$ satisfying the above properties are such that $w_{a_2, b_2}$ is not medium. Define $\mathcal{I}' := \mathcal{I} - \{(a_1, b_1)\}$. We now expand the product over the primes $p$ dividing $M'$. Then we bound the corresponding contribution by
\begin{multline}
\label{eSWbound}
\sum_{(\alpha_p)_{p \mid M'}} \sum_{(\mathbf{x}_p)_{p \mid M'}} \sum_{(v_a)_{a \in B - \{0\}}} \sum_{(w_{a, \bullet})} \sum_{(w_{a, b})_{(a, b) \in \mathcal{I'}}} \ell^{-m \sum_{(a, b) \in \mathcal{I}'} \tilde{\omega}(w_{a, b})}\\
\left|\sum_{\substack{\Delta(|w_{a_1, b_1}|) \leq \frac{X}{\prod_a \Delta(|v_a|) \prod_{(a, b) \in \mathcal{I}'} \Delta(|w_{a, b}|)} \\ \\ w_{a_1, b_1} \equiv c' \bmod d(\ell) \\ \gcd(w_{a_1, b_1}, M) = 1 \\ p \mid w_{a_1, b_1} \Rightarrow p \bmod M \in H_{a_1}}} \ell^{-m\tilde{\omega}(w_{a_1, b_1})} \left(\prod_{j \in [n]} \prod_{p \mid w_{a_1, b_1}} \delta_{a_1, p \bmod M, j, f}^{\pi_j(b_1)}\right) \times \right. \\
\left(\prod_{(a_2, b_2) \in \mathcal{I}'} \psi_{w_{a_2, b_2},\ell, 1}^{\langle b_1, a_2 \rangle}(\text{Frob}_{w_{a_1, b_1}}) \psi_{w_{a_1, b_1}, \ell, 1}^{\langle b_2, a_1 \rangle}(\text{Frob}_{w_{a_2, b_2}})\right) \times \\
\left. \left(\prod_{p \mid M'} \prod_{j \in [n]} \psi_{w_{a_1, b_1}, \ell, 1}(\text{Frob}_p)^{\pi_j(\mathbf{x}_p) \pi_j(a_1)} \overline{\chi_j(\alpha_p)}^{\pi_j(\mathbf{x}_p)}\right) \right|,
\end{multline}
where we also stipulate that the $w_{a, b}$ are squarefree, pairwise coprime and satisfy $p \mid w_{a, b} \Rightarrow p \equiv 0, 1 \bmod \ell$. We now define the multiplicative function $h$ supported on squarefree integers and given on the primes coprime to $\ell$ by
\begin{multline*}
h(q) = \ell^{-m} \times \mathbf{1}_{\gcd(q, M \prod_a v_a \prod_{(a, b) \in \mathcal{I}'} w_{a, b}) = 1} \times \mathbf{1}_{q \mod M \in H_{a_1}} \times \mathbf{1}_{q \equiv 1 \bmod \ell} \times \\
\prod_{j \in [n]} \delta_{a_1, q \bmod M, j, f}^{\pi_j(b_1)} \times \prod_{(a_2, b_2) \in \mathcal{I}'} \psi_{w_{a_2, b_2}, \ell, 1}^{\langle b_1, a_2 \rangle}(\text{Frob}_q) \psi_{q, \ell, 1}^{\langle b_2, a_1 \rangle}(\text{Frob}_{w_{a_2, b_2}}) \times \\
\prod_{p \mid M'} \prod_{j \in [n]} \psi_{q, \ell, 1}(\text{Frob}_p)^{\pi_j(\mathbf{x}_p) \pi_j(a_1)} \overline{\chi_j(\alpha_p)}^{\pi_j(\mathbf{x}_p)}.
\end{multline*}
We claim that
\begin{align}
\label{ehSW}
\sum_{1 \leq q \leq X} h(q) \log q = O_A\left(\frac{X}{(\log X)^A}\right)
\end{align}
for every $A > 0$. Applying Theorem \ref{thm:GK} then shows that
\[
\sum_{1 \leq n \leq X} h(n) = O_A\left(\frac{X}{(\log X)^A}\right)
\]
for every $A > 0$. Using this for a sufficiently large $A$ and inserting this into equation (\ref{eSWbound}) gives the desired upper bound for equation (\ref{eSWbound}) after a trivial summation.

It remains to prove equation (\ref{ehSW}). For now we assume that $\ell > 2$, and we will later sketch the modifications to get the case $\ell = 2$. Before we proceed, let us remark that the sum
\[
\sum_{1 \leq q \leq X} \psi_{q, \ell, 1}(\text{Frob}_{w_{a_2, b_2}})
\]
need not oscillate. Indeed, recall that $\psi_{q, \ell, 1}$ depends on our choice of $\sigma_q$ and for a dramatically poor choice we might (for example) have
\[
\psi_{q, \ell, 1}(\text{Frob}_{w_{a_2, b_2}}) \in \{1, \zeta_{\ell}\}.
\]
With this in mind, let us now work towards the proof of equation (\ref{ehSW}). We now pass to $\Q(\zeta_{\ell})$. We have already seen in Lemma \ref{lChoice} that the choice of $\psi_{q, \ell, 1}$ is equivalent to a choice of prime ideal $\mathfrak{p}$ of $\Q(\zeta_{\ell})$ above $q$. We also remind the reader that this ideal was called $\text{Pref}_\ell(q)$. Consider the Hecke character $\rho$ of $\Q(\zeta_{\ell})$ defined as
\begin{multline*}
\rho(\pp) = \mathbf{1}_{\gcd(N_{\Q(\zeta_{\ell})/\Q}(\pp), M \prod_a v_a \prod_{(a, b) \in \mathcal{I}'} w_{a, b}) = 1} \times \\
\prod_{(a_2, b_2) \in \mathcal{I}'} \left(\frac{N_{\Q(\zeta_{\ell})/\Q}(\pp)}{\text{Pref}_\ell(w_{a_2, b_2})}\right)_{\Q(\zeta_{\ell}), \ell}^{\langle b_1, a_2 \rangle} \left(\frac{w_{a_2, b_2}}{\pp}\right)_{\Q(\zeta_{\ell}), \ell}^{\langle b_2, a_1 \rangle} \times \prod_{p \mid M'} \prod_{j \in [n]} \left(\frac{p}{\mathfrak{p}}\right)_{\Q(\zeta_{\ell}), \ell}^{\pi_j(\mathbf{x}_p) \pi_j(a_1)}.
\end{multline*}
We have the fundamental identity
\begin{align}
\label{ehPref}
\ell^m \cdot h(q) = \mathbf{1}_{q \bmod M \in H_{a_1}} \cdot \rho(\text{Pref}_\ell(q)) \cdot \zeta(q),
\end{align}
where $\zeta(q)$ is an $\ell$-th root of unity depending only on $q \bmod M$. Since $|w_{a_2, b_2}| > 2$ and since $w_{a_2, b_2}$ is coprime to $M$, we see that $\rho$ is a non-trivial character and so is $\rho \chi$ for any Dirichlet character $\chi$ modulo $M$. Then the main theorem of Goldstein \cite{Goldstein} yields
\begin{align}
\label{eSWOsc}
\sum_{N_{\Q(\zeta_{\ell})/\Q}(\mathfrak{p}) \leq X} (\rho \chi)(\mathfrak{p}) = O_A\left(\frac{X}{(\log X)^A}\right)
\end{align}
for every $A > 0$. Note that Goldstein's result formally only applies to primitive characters, but one readily passes to non-primitive characters by trivially bounding the contribution of the primes dividing the conductor. Here we use that the norm of the conductor is bounded by a suitable power of $\log X$ depending only on our starting abelian group.

Our goal is now to deduce equation (\ref{ehSW}) from equation (\ref{ehPref}) and equation (\ref{eSWOsc}). But as we have emphasised before, this may not be possible if we made a dramatically poor choice of $\sigma_q$. We say that a choice $(\text{Pref}_\ell(q))_{q \equiv 1 \bmod \ell, q \leq X}$ is poor if there exists some integer $\exp((\log X)^{A_1}) \leq n \leq X$, some Hecke character $\psi$ of $\Q(\zeta_\ell)$ of order $\ell$ such that the norm of the conductor is bounded by $X$ and some $a \in \mathbb{F}_\ell^\ast$ such that
\begin{align}
\label{eBad}
\left||\{q \leq n : \psi(\text{Pref}_\ell(q)) = \zeta_\ell^a\}| - \frac{|\{q \leq n : \psi(\text{Pref}_\ell(q)) \neq 1\}|}{\ell - 1}\right| \geq n^{3/4}.
\end{align}
For fixed $n$, $\psi$ and $a$, observe that this is an entirely combinatorial condition. Indeed, $\psi(\text{Pref}_\ell(q))$ runs through all values of $\zeta_\ell^a$ with $a \in \mathbb{F}_\ell^\ast$ exactly once as we run through the choices of $\text{Pref}_\ell(q)$. For fixed $n$, $\psi$ and $a$, we bound the event (\ref{eBad}) using Hoeffding's inequality with the probability space corresponding to the set of choices $(\text{Pref}_\ell(q))_{q \equiv 1 \bmod \ell, q \leq X}$. This also gives a bound for the event that $(\text{Pref}_\ell(q))_{q \equiv 1 \bmod \ell, q \leq X}$ is poor by using the union bound. This shows that for $X$ large enough, we may pick $\sigma_q$ such that $(\text{Pref}_\ell(q))_{q \equiv 1 \bmod \ell, q \leq X}$ is not poor. In particular, we may pick one such choice of $\sigma_q$ at the start of our proof. Then we have that
\[
\sum_{1 \leq q \leq X} (\rho \chi)(\text{Pref}_\ell(q)) = O_A\left(\frac{X}{(\log X)^A}\right),
\]
which implies that
\[
\sum_{\substack{1 \leq q \leq X \\ q \equiv a \bmod M}} \rho(\text{Pref}_\ell(q)) = O_A\left(\frac{X}{(\log X)^A}\right)
\]
for every invertible class $a \bmod M$.

Thanks to the above equation and equation (\ref{ehPref}), we obtain that
\[
\ell^m \cdot \sum_{1 \leq  q \leq X} h(q) = \sum_{\substack{\lambda \in (\Z/M\Z)^\ast \\ \lambda \in H_{a_1}}} \zeta(\lambda) \sum_{\substack{1 \leq q \leq X \\ q \equiv \lambda \bmod M}} \rho(\text{Pref}_\ell(q)) = O_A\left(\frac{X}{(\log X)^A}\right).
\]
We deduce that equation (\ref{ehSW}) holds by partial summation.

In the case $\ell = 2$, we must also contend with the congruence condition $w_{a_1, b_1} \equiv c' \bmod 16$. We detect this congruence condition using Dirichlet characters and may now proceed as above.

\subsection{The main term}
We now use the results from Section \ref{sComb} to finish the proof. We distinguish two cases.

Let us start with the case $\ell > 2$. We apply Lemma \ref{lBlock2} with
\[
Y = \{(a, b) \in \mathcal{I} : w_{a, b} \text{ large}\}.
\]
By construction of $N_{\text{main}}(X)$ we have that
\begin{itemize}
\item $|Y| \geq \ell^n - |B|$ thanks to the material in Subsection \ref{ssLarge};
\item $\langle a_2, b_1 \rangle = 0$ for all $(a_1, b_1), (a_2, b_2) \in Y$ thanks to the material in Subsection \ref{ssLinked}.
\end{itemize}
Therefore all conditions of Lemma \ref{lBlock2} are satisfied. We conclude that
\[
Y = \{(a, 0) : a \in \mathbb{F}_\ell^n - B\}.
\]
We now observe that for every $(a_1, b_1) \in \mathcal{I}$ with $b_1 \neq 0$ there exists some $(a, 0) \in Y$ such that $\langle a, b_1 \rangle \neq 0$. By Subsection \ref{ssLinked} and by definition of $N_{\text{main}}(X)$, this forces $w_{a, b} = 1$ for all $b \neq 0$. Therefore $N_{\text{main}}(X)$ becomes
\begin{multline*}
N_{\text{main}}(X) = \sum_{\substack{(v_a)_{a \in B - \{0\}} \\ p \mid v_a \Rightarrow p \bmod M \in H_a}} \sum_{\substack{(w_{a, 0})_a \\ w_{a, 0} \text{ large} \\ \prod_a \Delta(|v_a|) \prod_a \Delta(|w_{a, 0}|) \leq X \\ v_a \equiv c_a \bmod d(\ell) \\ \gcd(w_{a, 0}, M) = 1 \\ p \mid w_{a, 0} \Rightarrow p \bmod M \in H_a}} \mathbf{1}_{\mathbf{v} \in \mathcal{A}} \times \ell^{-m \sum_{a, 0} \tilde{\omega}(w_{a, 0})} \times \\ 
\prod_{p \mid M'} \left(\frac{1}{\ell^n} \sum_{\alpha \in g(p)} \sum_{\mathbf{x} \in \mathbb{F}_\ell^n} \prod_{j \in [n]} \psi_j(\text{Frob}_p)^{\pi_j(\mathbf{x})} \overline{\chi_j(\alpha)}^{\pi_j(\mathbf{x})}\right),
\end{multline*}
where we also demand that all $w_{a, 0}$ are large. Expanding the product over $M'$, we get a combination of non-principal characters of small conductor unless $\mathbf{x}$ is the zero vector. Another application of Siegel--Walfisz, where we sum over an appropriate variable $w_{a, 0}$ depending on $\mathbf{x}$, yields the asymptotic
\[
N_{\text{main}}(X) = \sum_{\substack{(v_a)_{a \in B - \{0\}} \\ p \mid v_a \Rightarrow p \bmod M \in H_a}} \sum_{\substack{(w_{a, 0})_a \\ w_{a, 0} \text{ large} \\ \mathbf{v} \in \mathcal{A} \\ \prod_a \Delta(|v_a|) \prod_a \Delta(|w_{a, 0}|) \leq X \\ v_a \equiv c_a \bmod d(\ell) \\ \gcd(w_{a, 0}, M) = 1 \\ p \mid w_{a, 0} \Rightarrow p \bmod M \in H_a}} \hspace{-0.5cm} \ell^{-m \sum_{a, 0} \tilde{\omega}(w_{a, 0})} \prod_{p \mid M'} \frac{|g(p)|}{\ell^n} + O_A\left(\frac{X}{(\log X)^A}\right).
\]
We may also remove the condition that $w_{a, 0}$ is large with an acceptable error term. Therefore we conclude that
\[
N(X) = \sum_{\substack{(v_a)_{a \in B - \{0\}} \\ p \mid v_a \Rightarrow p \bmod M \in H_a}} \sum_{\substack{(w_{a, 0})_a \\ \mathbf{v} \in \mathcal{A} \\ \prod_a \Delta(|v_a|) \prod_a \Delta(|w_{a, 0}|) \leq X \\ v_a \equiv c_a \bmod d(\ell) \\ \gcd(w_{a, 0}, M) = 1 \\ p \mid w_{a, 0} \Rightarrow p \bmod M \in H_a}} \hspace{-0.75cm} \ell^{-m \sum_{a, 0} \tilde{\omega}(w_{a, 0})} \prod_{p \mid M'} \frac{|g(p)|}{\ell^n} + O_A\left(X (\log X)^{\alpha - 1- \delta}\right)
\]
for some $\delta > 0$. Since $\ell$ is odd, we have $d(\ell) = 1$. We directly evaluate the above sum using Theorem \ref{thm:GK} with
\begin{multline*}
\alpha = \sum_{a \in B - \{0\}} \frac{1}{\varphi(\text{lcm}(M, \ell))} \sum_{\lambda \in (\Z/\text{lcm}(M, \ell)\Z)^\ast} \mathbf{1}_{\lambda \equiv 1 \bmod \ell} \cdot \mathbf{1}_{\lambda \bmod M \in H_a} + \\
\frac{1}{\ell^m} \sum_{a \in A - B} \frac{1}{\varphi(\text{lcm}(M, \ell))} \sum_{\lambda \in (\Z/\text{lcm}(M, \ell)\Z)^\ast} \mathbf{1}_{\lambda \equiv 1 \bmod \ell} \cdot \mathbf{1}_{\lambda \bmod M \in H_a},
\end{multline*}
which is readily verified to be the correct exponent for the logarithm. One also finds that the leading constant equals the conditionally convergent product
\begin{multline*}
C_{\text{lead}} = \prod_{\substack{p \equiv 1 \bmod \ell \\ \gcd(p, M) = 1}} \left(1 + \frac{\sum_{\substack{a \in B - \{0\} \\ p \bmod M \in H_a}} 1 + \sum_{\substack{a \in A - B \\ p \bmod M \in H_a}} 1/\ell^m}{p}\right) \cdot \prod_p \left(1 - \frac{1}{p}\right)^\alpha \cdot \\ 
\left(\prod_{p \mid M'} \frac{|g(p)|}{\ell^n}\right) \cdot \frac{1}{\Gamma(\alpha)} \cdot \left(\sum_{\substack{(d_a)_{a \in A - \{0\}} \\ d_a \mid \ell \\ \gcd(d_a, M) = 1}} \frac{\mu^2\left(\prod_{a \in A - \{0\}} d_a\right)}{\prod_{a \in A - \{0\}} d_a^2}\right). 
\end{multline*}
In particular, if $\ell > 2$, we always have $C_{\text{lead}} > 0$. It is also clear from the above expression that $C_{\text{lead}}$ is uniformly bounded.

Let us now see how to modify the above argument for $\ell = 2$. In this case we apply Lemma \ref{lBlock} with
\[
Y = \{(a, b) \in \mathcal{I} : w_{a, b} \text{ large}\}.
\]
By construction of $N_{\text{main}}(X)$ we have that
\begin{itemize}
\item $|Y| \geq 2^n - |B|$ due to Subsection \ref{ssLarge};
\item $\langle a_1, b_2 \rangle + \langle a_2, b_1 \rangle = 0$ for all $(a_1, b_1), (a_2, b_2) \in Y$ due to Subsection \ref{ssLinked}.
\end{itemize}
Lemma \ref{lBlock} yields
\[
Y = \{(a, f(a)) : a \in \mathbb{F}_2^n\}
\]
for some alternating map $f$. The case $B = \mathbb{F}_2^n$ is easy, so let us suppose that $B$ is a proper subspace of $\FF_2^n$. We now observe that for every $(a_1, b_1) \in \mathcal{I}$ with $b_1 \neq f(a_1)$ there exists some $(a, f(a)) \in Y$ such that 
\[
\langle a_1, f(a) \rangle + \langle a, b_1 \rangle \neq 0.
\]
Indeed, take any $a \not \in B$ such that $\langle a, f(a_1) - b_1 \rangle \neq 0$.

By Subsection \ref{ssLinked} and by definition of $N_{\text{main}}(X)$, this forces $|w_{a, b}| \leq 2$ for all $(a, b) \in \mathcal{I} - Y$ and therefore $w_{a, b} \in \{-2, -1, 1, 2\}$ and thus $w_{a, b} = 1$ unless $b = \bullet$. We now analyze $N_{\text{main}}(X)$ by splitting over congruence conditions on $w_{a, f(a)}$ modulo $16M$ and the parity of the number of prime divisors $p$ such that $p \bmod M$ lies in a given coset of $G$. Crucially, this fixes both
\[
\prod_{j \in [n]} \prod_{p \mid w_{a_1, f(a_1)}} \delta_{a_1, p \bmod M, j, f}^{\pi_j(f(a_1))}
\]
and
\[
\prod_{p \mid M'} \left(\frac{1}{2^n} \sum_{\alpha \in g(p)} \sum_{\mathbf{x} \in \mathbb{F}_2^n} \prod_{j \in [n]} \psi_j(\text{Frob}_p)^{\pi_j(\mathbf{x})} \overline{\chi_j(\alpha)}^{\pi_j(\mathbf{x})}\right).
\]
Therefore, using quadratic reciprocity and the values of $c_a$ modulo $16$ to also eliminate the terms of the shape
\[
\psi_{w_{a_2, b_2}, \ell, 1}(\text{Frob}_{w_{a_1, b_1}}) \psi_{w_{a_1, b_1}, \ell, 1}(\text{Frob}_{w_{a_2, b_2}}),
\]
our main term $N_{\text{main}}(X)$ becomes
\[
N_{\text{main}}(X) = \sum_f \sum_{(c_a)} \sum_{(e_{a, \lambda})} \sum_{\substack{(v_a)_{a \in B - \{0\}} \\ p \mid v_a \Rightarrow p \bmod M \in H_a}} \epsilon((c_a), (e_a)) \hspace{-0.3cm} \sum_{\substack{(w_{a, f(a)})_a \\ \mathbf{v} \in \mathcal{A} \\ \prod_a \Delta(|v_a|) \prod_a \Delta(|w_{a, f(a)}|) \leq X \\ v_a \equiv c_a \bmod 16M \\ \gcd(w_{a, f(a)}, M) = 1 \\ p \mid w_{a, f(a)} \Rightarrow p \bmod M \in H_a \\ \omega_\lambda(v_a) \equiv e_{a, \lambda} \bmod 2}} 2^{-m \sum_a \tilde{\omega}(w_{a, f(a)})},
\]
where $\epsilon((c_a), (e_a))$ is a real number, bounded in absolute value by $1$, depending only on $c_a$ and $e_a$, where $\omega_\lambda(v_a)$ denotes the number of prime divisors $p$ of $v_a$ such that $p \bmod M$ is in the coset $\lambda + G$ and where the sum over $f$ is over all alternating maps. We now detect the congruence condition $v_a \equiv c_a \bmod 16M$ using Dirichlet characters and the condition $\omega_\lambda(v_a) \equiv e_{a, \lambda} \bmod 2$ using
\[
\frac{1}{2}\left(1 + (-1)^{e_{a, \lambda} + \omega_\lambda(v_a)}\right).
\]
To finish the proof, we proceed as in the case $\ell = 2$ to get the asymptotic formula for $N_{\text{main}}(X)$ by several applications of Theorem \ref{thm:GK}.

Let us now check the final part of Theorem \ref{tMultiReduction}. It is still readily verified that $C_{\text{lead}}$ is bounded in terms of $n$ only. For the final part, one gets that $C_{\text{lead}} > 0$ by directly adapting the above argument for $\ell > 2$ (including the application of Siegel--Walfisz to get the main term for the conditions on the primes $p$ dividing $M'$) observing that
\[
\delta_{a_1, p \bmod M, j, f}^{\pi_j(f(a_1))} = 1
\]
and
\[
\psi_{w_{a_2, b_2}, 2, 1}(\text{Frob}_{w_{a_1, b_1}}) \psi_{w_{a_1, b_1}, 2, 1}(\text{Frob}_{w_{a_2, b_2}}) = 1
\]
under the assumptions of the final part of Theorem \ref{tMultiReduction}.

\section{Explicit constants}
\label{sec:explicit}
In general, the leading constant in Theorem \ref{thm:main} is an infinite sum of Euler products, which is unlikely to be expressible in a simple fashion. However, in specific cases it is possible to write down an explicit description of the leading constant. In particular, for purely multicyclic extensions, since $C_{\text{lead}}$ in Theorem \ref{tMultiReduction} can be made explicit so can the leading constant in the count of extensions satisfying the weak approximation property. As a further example, in this section we will also compute the leading constant for $C_2 \times C_3 \times C_3$ as this is the smallest group $A$ for which a positive proportion of fields, when ordered by discriminant, have the weak approximation property.

\subsection{Multicyclic Extensions}
\label{sec:multicyc}
Before writing down the explicit asymptotic for the number of $\mathbb{F}_\ell^n$-extensions of $\Q$ whose norm one torus satisfies weak approximation, we start by explicitly enumerating the $\mathbb{F}_\ell^n$-extensions of bounded absolute discriminant, for the purpose of comparison. While the asymptotic formula, with an inexplicit constant, dates back to the work of Wright~\cite{Wright}, to the authors' knowledge, the only examples where the leading constant is explicitly known are the cases:
\begin{itemize}
\item $\ell =2$ (due independently to la Bret{\`e}che--Kurlberg--Shparlinski~\cite{lBKS} and Fritsch~\cite{Fritsch}), 
\item $\ell =3$, $n=2$ (due to Mammo~\cite{Mammo}),
\item $n=1$, for any prime $\ell$ over any base field (due to Cohen--Diaz-y-Diaz--Olivier~\cite{cyclic}).
\end{itemize}
For a survey of some of the results known about explicit constants in Malle's conjecture, see Cohen--Diaz-y-Diaz--Olivier~\cite{constants}.
We restrict attention to the case $\ell>2$, to avoid the additional complications of the oddest prime, and $n \geq 2$, since weak approximation always holds on the norm one torus of a cyclic extension. In this setting, all but the simplest case of the following theorem is new.

\begin{theorem}
\label{thm:multicycWright}
Let $\ell>2$ be a prime and $n>0$ an integer. For a real number $X>1$, denote by $N_{\ell,n}(X)$ the number of field extensions $K/\Q$ such that $\Gal(K/\Q) \cong \mathbb{F}_\ell^n$ and $\textup{Disc}(K/\Q) \leq X$. Then, as $X \rightarrow \infty$, we have
\[
N_{\ell,n}(X) = \frac{C}{\Gamma \left(\frac{\ell^n-1}{\ell-1} \right)} X^{\frac{1}{\ell^n - \ell^{n-1}}} (\log X)^{\frac{\ell^n-1}{\ell-1}-1} + O(X^{\frac{1}{\ell^n - \ell^{n-1}}} (\log X)^{\frac{\ell^n-1}{\ell-1}-2}),
\]
where
\[
C
=
\frac{1}{\prod_{i=0}^{n-1} (\ell^n - \ell^i)} \frac{\left( 1 + \frac{\ell^n-1}{\ell^2} \right)}{(\ell^n - \ell^{n-1})^{\frac{\ell^n -1}{\ell -1} -1}} \prod_{p \equiv 1 \bmod \ell} \left( 1 + \frac{\ell^n -1}{p} \right) \prod_p \left( 1 - \frac{1}{p} \right)^{\frac{\ell^n-1}{\ell-1}}.
\]
\end{theorem}

\begin{remark}
The authors are fairly confident that one could prove an asymptotic with an explicit degree $\frac{\ell^n-1}{\ell-1}-1$ polynomial in $\log X$ and power saving error term by replacing the application of Theorem \ref{thm:GK} in the proof with a customised application of the Selberg--Delange approach (perhaps multidimensional in nature, as in \cite{lBKS}). 
\end{remark}

\begin{proof}
To begin, we parametrise $\mathbb{F}_\ell^n$-extensions of $\Q$ as in Section \ref{sec:par}. That is to say, each such extension corresponds to a tuple $(v_a)_{a \in \mathbb{F}_\ell^n - \{ 0\}}$, where the entries are squarefree and pairwise coprime. Furthermore, we stipulate that the $v_a$ must satisfy
\begin{itemize}
\item we have
\[
p \equiv 1 \bmod \frac{\ord(a)}{p^{v_p(\ord(a))}}
\]
for all prime divisors $p$ of $v_a$;
\item if $\textup{ord}(a) > 2$, then $v_a > 0$.
\end{itemize}
The latter condition forces each entry in the tuple to be positive, while the former requires each prime divisor of $v_a$ which is not equal to $\ell$ to be congruent to 1 modulo $\ell$. Finally, we observe that the discriminant of the field associated to such a tuple is given by
\[
\begin{cases}
\prod v_a^{\ell^n - \ell^{n-1}} \quad{} &\text{ if } \ell \nmid v_a \text{ for all } a,\\
\ell^{2(\ell^n - \ell^{n-1})} \prod (v_a')^{\ell^n - \ell^{n-1}}  \quad{} &\text{ if } \exists a \text{ s.t. } \ell \mid v_a,
\end{cases}\]
where $v_a' = v_a \ell^{-v_{\ell}(v_a)}$.

Note that the parametrisation just discussed does not parametrise $\mathbb{F}_\ell^n$-extensions of $\Q$ but rather homomorphisms from $G_{\Q}$ to $\mathbb{F}_\ell^n$. However, for each field there are several such homomorphisms. In order to correct this overcount, we need to divide by the size of the automorphism group of $\mathbb{F}_\ell^n$.

Hence, splitting up the two cases where $\ell \nmid \prod v_a$ and $\ell \mid \prod v_a$, we find that the count for $N_{\ell, n}(X)$ may be given by the sum
\[
\frac{1}{|\Aut(\mathbb{F}_\ell^n)|}\left[
\sum_{\substack{\mathbf{v} \in \Z^{\ell^n-1}_{>0} \\ p \mid v_i \Rightarrow p \equiv 1 \bmod \ell \\ \gcd(v_i, 2 \ell) =1 \\ \prod_i v_i^{\ell^n - \ell^{n-1}} \leq X }} \mu^2\left( \prod_i v_i \right)
+
(\ell^n-1)
\sum_{\substack{\mathbf{v} \in \Z^{\ell^n-1}_{>0}\\ p \mid v_i \Rightarrow p \equiv 1 \bmod \ell \\ \gcd(v_i, 2 \ell) =1 \\ \prod_i v_i^{\ell^n - \ell^{n-1}} \leq X/\ell^{2( \ell^n - \ell^{n-1}}) }} \mu^2\left( \prod_i v_i \right) \right].
\]
We can compute the size of the sums appearing here using Theorem \ref{thm:GK}. Indeed, we may write for any $V$
\[
\sum_{\substack{\mathbf{v} \in \Z^{\ell^n-1}_{>0}\\ p \mid v_i \Rightarrow p \equiv 1 \bmod \ell \\ \gcd(v_i, 2 \ell) =1 \\ \prod_i v_i \leq V }} \mu^2\left( \prod_i v_i \right)
=
\sum_{\substack{ v \in \Z_{>0} \\ v \leq V}} \mu^2(v) g(v),
\]
where $g(v)$ is the $\ell^n-1$ fold convolution of the indicator function $\mathbf{1}_{p \mid v \Rightarrow p \equiv 1 \bmod \ell}(v)$. Note that we have dropped the coprimality condition $\gcd(v_i, 2\ell) = 1$, since it is already implied by the constraint that prime divisors of $v$ must be congruent to 1 mod $\ell$. On primes, we can estimate the sum
\begin{align*}
\sum_{p \leq x} g(p) \log p &= \sum_{p \leq x} ( \ell^n -1) \mathbf{1}_{p \mid v \Rightarrow p \equiv 1 \bmod \ell}(v) \log p\\ &= (\ell^n-1) \sum_{\substack{p \leq x \\ p \equiv 1 \bmod \ell}} \log p
= \frac{\ell^n-1}{\ell-1} x + O_A\left( \frac{x}{(\log x)^A} \right),
\end{align*} by the classical Siegel--Walfisz theorem. Therefore, applying Theorem \ref{thm:GK}, we have
\[
\sum_{\substack{ v \in \Z_{>0} \\ v \leq V}}\! \mu^2(v) g(v)
=
\frac{V (\log V)^{\frac{\ell^n-1}{\ell-1}-1}}{\Gamma \left(\frac{\ell^n-1}{\ell-1} \right)} \prod_{p \equiv 1 \bmod \ell}\! \left( 1 + \frac{\ell^n -1}{p} \right) \prod_p\! \left( 1 - \frac{1}{p} \right)^{\frac{\ell^n-1}{\ell-1}} \!\!
+ 
O\!\left(\! V (\log V)^{\frac{\ell^n-1}{\ell-1}-2}\! \right)\!.
\]
Combining this with our previous expression for $N_{\ell,n}(X)$ completes the claim.
\end{proof}

\begin{remark}
To compare with Mammo's result, set $\ell = 3$ and $n=2$. Then the main term for $N_{3,2}(X)$ is given by
\[
\frac{1}{(3^2-1)(3^2-3)}
\left( 1 + \frac{3^2 -1}{3^2}\right)
\frac{X^{\frac{1}{6}} (\log X)^3}{6^3 \Gamma(4)} \prod_{p \equiv 1 \bmod 3} \left(1 + \frac{8}{p} \right) \prod_p \left( 1 - \frac{1}{p} \right)^4
\]\[=\frac{1}{48}\frac{17}{9}
\frac{X^{\frac{1}{6}} (\log X)^3}{6^4} \left( \frac{2}{3} \right)^4\prod_{p \equiv 1 \bmod 3} \left(1 + \frac{8}{p} \right)\left( 1 - \frac{1}{p} \right)^4 \prod_{p \equiv 2 \bmod 3} \left( 1 - \frac{1}{p} \right)^4.
\]
This Euler product can be simplified by comparison with $L(1, \chi)^4$ for $\chi$ the non-principal Dirichlet character mod 3. Thus our leading constant becomes
\[\frac{17X^{\frac{1}{6}} (\log X)^3}{2^43^{11}} L(1, \chi)^4 \prod_{p \equiv 1 \bmod 3} \left(1 + \frac{8}{p} \right)\left( 1 - \frac{1}{p} \right)^8 \prod_{p \equiv 2 \bmod 3} \left( 1 - \frac{1}{p^2} \right)^4.\]
By the analytic class number formula, $L(1, \chi) = \frac{2 \pi}{6\sqrt{3}}$. This means that we have
\[
N_{3,2}(X) \sim
\frac{17\pi^4X^{\frac{1}{6}} (\log X)^3}{2^43^{17}}  \prod_{p \equiv 1 \bmod 3} \left(1 + \frac{8}{p} \right)\left( 1 - \frac{1}{p} \right)^8 \prod_{p \equiv 2 \bmod 3} \left( 1 - \frac{1}{p^2} \right)^4,
\]
which recovers the result of \cite[$\S$ 7, Case 2]{Mammo}.
Since the proof in \cite{Mammo} relies on an application of the Ikehara--Delange Tauberian theorem there is no error term given so our result represents an improvement in that aspect.
\end{remark}

We now turn to the problem of weak approximation on the norm one torus. Our count will be provided by Theorem \ref{tMultiReduction}, however many of the complications arising in the case of general abelian extensions are not necessary for multicyclic extensions. In particular, we may set $M=1$ which makes the subgroups $H_a$ trivial and we have by definition that $d(\ell) =1$. Note also that $m := n-1 - \text{dim}_{ \mathbb{F}_{\ell}} B = n-1$. Thus, by the work of the previous section in concluding the proof of Theorem \ref{tMultiReduction}, the number of $\mathbb{F}_\ell^n$-extensions of $\Q$ with absolute discriminant at most $X$ whose norm one torus satisfies weak approximation is equal to
\[
\frac{1}{\prod_{i=0}^{n-1} (\ell^n - \ell^i)} C_{\text{lead}} X^{\frac{1}{\ell^n - \ell^{n-1}}} \log(X^{\frac{1}{\ell^n - \ell^{n-1}}} )^{\alpha -1} + O \left(X^{\frac{1}{\ell^n - \ell^{n-1}}}  (\log X)^{\alpha - 1 - \delta} \right)
\]
for some $\delta > 0$, where
\begin{align*}
C_{\text{lead}} &= \prod_{p \equiv 1 \bmod \ell} \left( 1 + \frac{\sum_{a \in \mathbb{F}_\ell^n - \{0\}} \ell^{-n}}{p}\right) 
\prod_p \left( 1 - \frac{1}{p} \right)^{\alpha} \frac{1}{\Gamma(\alpha)} 
\left( \sum_{\substack{(d_a)_{a \in \mathbb{F}_\ell^n-\{0\}} \\ d_a \mid \ell}}
\frac{\mu^2(\prod_a d_a)}{\prod_a d_a^2} 
\right)
,\\
\alpha &= 0 + \frac{1}{\ell^{n-1}} \sum_{a \in \mathbb{F}_\ell^n - \{0\}} \frac{1}{\varphi(\ell)} \sum_{\lambda \in (\Z/\ell \Z)^*} \mathbf{1}_{\lambda \equiv 1 \bmod \ell}
=\frac{\ell^n-1}{\ell^{n-1}(\ell-1)}.
\end{align*}
We observe that as in the proof of Theorem \ref{thm:multicycWright}, we have divided by a factor of $|\Aut(\mathbb{F}_\ell^n)| = \prod_{i=0}^{n-1} (\ell^n - \ell^i)$. Using the fact that
$
\left( \sum_{\substack{(d_a)_{a \in \mathbb{F}_\ell^n-\{0\}} \\ d_a \mid \ell}}
\frac{\mu^2(\prod_a d_a)}{\prod_a d_a^2} 
\right)
=
1 + \frac{\ell^n-1}{\ell^2},
$
the main term becomes
\begin{multline*}
\frac{X^{\frac{1}{\ell^n - \ell^{n-1}}} (\log X)^{\frac{\ell^n-1}{\ell^{n-1}(\ell-1)} -1 } }{\Gamma\left( \frac{\ell^n-1}{\ell^{n-1}(\ell-1)} \right)}
\frac{\left( 1 + \frac{\ell^n-1}{\ell^2} \right)}{(\ell^n - \ell^{n-1})^{\frac{\ell^n-1}{\ell^{n-1}(\ell-1)}-1}\prod_{i=0}^{n-1} (\ell^n - \ell^i)} \\ \times
\prod_{p \equiv 1 \bmod \ell} \left( 1 + \frac{\ell^n-1}{\ell^{n-1} p} \right) \prod_p \left( 1- \frac{1}{p} \right)^{\frac{\ell^n-1}{\ell^{n-1}(\ell-1)}}\!\!
\end{multline*}
as desired.

\subsection{A positive proportion of fields satisfying weak approximation}
In order to parametrise $A = C_2 \times C_3 \times C_3$-extensions, using the method of Section \ref{sec:par}, we need a 17-tuple of squarefree, pairwise coprime integers. The tuple is indexed by the non-identity elements of $C_2 \times C_3 \times C_3$ and hence we will write the tuple as $(u_1, \ldots, u_8, v_1, \ldots, v_8, w)$ where the $u_i$ are those components indexed by elements of order 3 in $C_2 \times C_3 \times C_3$, the $v_i$ are those components indexed by elements of order 6, and $w$ is the component indexed by the order 2 element. The $u_i$ and $v_i$ must be positive, but $w$ can be either positive or negative. To complete the parametrisation we must impose the further conditions that
\begin{itemize}
\item $p \mid u_i \Rightarrow p \equiv 0,1 \bmod 3$,
\item $p \mid v_i \Rightarrow p \equiv 1,3 \bmod 6$.
\end{itemize}
The proof will proceed by fixing the $u_i$ and $v_i$ and thus determining a $C_3 \times C_3$-extension $F$ of small discriminant, after which we will vary $w$. To ensure that the extension $K/\Q$ is such that the norm one torus satisfies weak approximation we must turn to the criteria in Section \ref{sec:crit}. The subset $2A \cap A[2]$ contains only the identity and therefore, by Theorem \ref{tLocalConditions}, once the $u_i$ and $v_i$ are chosen, there is no additional condition on $w$. In other words, $R^1_{K/\Q} \mathbb{G}_m$ satisfies weak approximation if and only if $R^1_{F/\Q} \mathbb{G}_m$ does. Hence, in the general set-up in Section \ref{sec:reduction}, we can take $H_a$ to be the full group $(\Z/M\Z)^*$. This means that in order to compute $A = C_2 \times C_3 \times C_3$-extensions of bounded absolute discriminant satisfying the weak approximation condition, we must consider sums of the form
\[
\sum_{\substack{ \mathbf{u},\mathbf{v} \in \Z^8_{>0} \\ \text{Par}(\mathbf{u}, \mathbf{v}) \text{ satisfies WA}\\ \text{disc}(\text{Par}(\mathbf{u}, \mathbf{v})) \leq (\log X)^{100}\\ p \mid u_i \Rightarrow p \equiv 0,1 \bmod 3\\  p \mid v_i \Rightarrow p \equiv 1,3 \bmod 6}}
\sum_{\substack{ w \in \Z_{\neq 0} \\  \text{disc}(\text{Par}(\mathbf{u}, \mathbf{v},w)) \leq X}} \mu^2 \left( w\prod_i u_i \prod_j v_j \right).
\]
Next, we describe the discriminant of the extension $K/\Q$. We observe that the valuation of the discriminant at the primes 3 and $p$ for $p$ any prime $>3$ are given by
\[
v_p(\Delta_{K/\Q}) =
\begin{cases}
0 \phantom{3}\text{ if } p \nmid u_1\cdots u_8v_1\cdots v_8 w,\\
9 \phantom{3}\text{ if } p \mid w,\\
12 \text{ if } p \mid u_i,\\
15 \text{ if } p \mid v_i.
\end{cases}
, \quad{} 
v_3(\Delta_{K/\Q}) =
\begin{cases}
0 \phantom{3}\text{ if  } 3 \nmid u_1\cdots u_8v_1\cdots v_8 w,\\
9 \phantom{3}\text{ if  } p \mid w,\\
24 \text{ if } p \mid u_i,\\
27 \text{ if } p \mid v_i.
\end{cases}
\]
Finally, at 2 (noting that 2 never ramifies in a $C_3 \times C_3$-extension), we find that the discriminant has valuation
\[
v_2(\Delta_{K/\Q})
=
\begin{cases}
0 \phantom{3}\text{ if } w \equiv 1 \bmod 4,\\
18 \text{ if } w \equiv 3 \bmod 4,\\
27 \text{ if } w \equiv 2 \bmod 4.
\end{cases}
\]
The innermost sum will be handled using the well known estimate for odd $d$
\begin{equation}
\label{eq:sqf}
\sum_{\substack{ \gcd(w,d)=1\\ w \equiv a \bmod 4 \\ 0<w^9 \leq W}} \mu^2(w)
=
\frac{3}{\pi^2}W^{1/9} \prod_{p \mid 2d} \left( 1 + \frac{1}{p} \right)^{-1} + O\left(W^{1/18} \tau(d) \right).
\end{equation}
To compute the count for $A$-extensions whose norm one torus satisfies weak approximation, we will break into four cases:
\begin{enumerate}
\item $3 \nmid u_1\cdots u_8v_1\cdots v_8 w$,
\item there is some $i$ such that $3 \mid u_i$,
\item there is some $i$ such that $3 \mid v_i$,
\item $3 \mid w$.
\end{enumerate}
Note that these cases are both disjoint and exhaustive. We start with case 1. In this case, the sum that has to be computed can be written as
\[
\Sigma_1(X) :=
\sum_{\substack{ \mathbf{u},\mathbf{v} \in \Z^8_{>0} \\ \text{Par}(\mathbf{u}, \mathbf{v}) \text{ satisfies WA}\\ \prod_i u_i^{12} \prod_j v_j^{15} \leq (\log X)^{100}\\ p \mid u_i \Rightarrow p \equiv 1 \bmod 3\\ p \mid v_i \Rightarrow p \equiv 1 \bmod 6}}
\sum_{a \in \{1,2,3\}}
\sum_{\substack{ w \in \Z_{\neq 0} \\ w \equiv a \bmod 4\\ (c(a)\vert w\vert)^9 \leq X/(\prod_i u_i^{12} \prod_j v_j^{15})}} \mu^2 \left( 3w\prod_i u_i \prod_j v_j \right),
\]
where $c(1) = 1, c(2) = 4$ and $c(3) = 4$. We can express this as
\[
\Sigma_1(X) =
\sum_{\substack{ \mathbf{u},\mathbf{v} \in \Z^8_{>0} \\ \text{Par}(\mathbf{u}, \mathbf{v}) \text{ satisfies WA}\\ \prod_i u_i^{12} \prod_j v_j^{15} \leq (\log X)^{100}\\ p \mid u_i \Rightarrow p \equiv 1 \bmod 3\\ p \mid v_i \Rightarrow p \equiv 1 \bmod 6}} \mu^2\left(\prod_i u_i \prod_j v_j \right)
\sum_{a \in \{1,2,3\}}
\sum_{\substack{ w \in \Z_{\neq 0} \\ w \equiv a \bmod 4\\ \vert w \vert \leq (c(a))^{-1}(X/(\prod_i u_i^{12} \prod_j v_j^{15}))^{1/9}\\ \gcd\left( w , 3\prod_i u_i \prod_j v_j\right) = 1}} \mu^2 ( w).
\]
Applying \eqref{eq:sqf} to the inner sums over $a$ and $w$, we see that they are equal to
\[
2\times\frac{3}{\pi^2} \left(\frac{X}{\prod_i u_i^{12} \prod_j v_j^{15}}\right)^{\frac{1}{9}} \left(1+ \frac{1}{4} + \frac{1}{4} \right) \prod_{p \mid 6 \prod_i u_i \prod_j v_j} \left( 1 + \frac{1}{p} \right)^{-1} + O_{\epsilon}\left( \left( \frac{X}{\prod_i u_i^{12} \prod_j v_j^{15}}\right)^{\frac{1}{18} + \epsilon} \right).
\]
Since the $u_i$ and $v_j$ all have log power size, this error term is negligible when summed over the remaining variables. Thus
\[
\Sigma_1(X) \sim \frac{9}{\pi^2} X^{\frac{1}{9}}
\sum_{\substack{ \mathbf{u},\mathbf{v} \in \Z^8_{>0} \\ \text{Par}(\mathbf{u}, \mathbf{v}) \text{ satisfies WA}\\ p \mid u_i \Rightarrow p \equiv 1 \bmod 3\\ p \mid v_i \Rightarrow p \equiv 1 \bmod 6}} \frac{\mu^2\left(\prod_i u_i \prod_j v_j \right)}{\prod_i u_i^{\frac{4}{3}} \prod_j v_j^{\frac{5}{3}}}\prod_{p \mid \prod_i u_i \prod_j v_j} \left( 1 + \frac{1}{p} \right)^{-1} .
\]
Unfortunately, at this stage it seems that no further simplification is possible. The condition $\text{Par}(\mathbf{u}, \mathbf{v})$ satisfies WA is not a multiplicative condition on the variables $\mathbf{u}$ and $\mathbf{v}$ and thus this convergent sum cannot be expressed as an Euler product. One could detect this condition using character sums as discussed in Section \ref{sec:charsum} but it is unclear that doing so will give a path towards writing the sum in any kind of simpler form. From here on out, we will simply refer to this infinite sum as $\kappa$.

We now move to case 2. Suppose that $3 \mid u_{i_0}$ and write $\widetilde{u}_i = u_i$ for all $i \neq i_0$ and $\widetilde{u_{i_0}} = \frac{u_{i_0}}{3}$. Then the sum which has to be computed is
\[
\Sigma_2(X) := 
\sum_{\substack{ \widetilde{\mathbf{u}},\mathbf{v} \in \Z^8_{>0} \\ \text{Par}(\mathbf{u}, \mathbf{v}) \text{ satisfies WA}\\ \prod_i \widetilde{u}_i^{12} \prod_j v_j^{15} \leq (\log X)^{100}\\ p \mid \widetilde{u}_i \Rightarrow p \equiv 1 \bmod 3\\ p \mid v_i \Rightarrow p \equiv 1 \bmod 6}}
\sum_{a \in \{1,2,3\}}
\sum_{\substack{ w \in \Z_{\neq 0} \\ w \equiv a \bmod 4\\ (c(a)\vert w\vert)^9 \leq X/3^{24}(\prod_i \widetilde{u}_i^{12} \prod_j v_j^{15})}} \mu^2 \left( 3w\prod_i \widetilde{u}_i \prod_j v_j \right).
\]
Evidently, we have $\Sigma_2 = 3^{-\frac{24}{9}} \Sigma_1$. The argument in case 3 is almost exactly the same. We find that the sum we need to compute in this setting, $\Sigma_3(X)$ satisfies $\Sigma_3 = 3^{-\frac{27}{9}} \Sigma_1$. Finally, in case 4, write $\widetilde{w} = \frac{w}{3}$. Then we consider
\[
\Sigma_4(X) :=
\sum_{\substack{ \mathbf{u},\mathbf{v} \in \Z^8_{>0} \\ \text{Par}(\mathbf{u}, \mathbf{v}) \text{ satisfies WA}\\ \prod_i u_i^{12} \prod_j v_j^{15} \leq (\log X)^{100}\\ p \mid u_i \Rightarrow p \equiv 1 \bmod 3\\ p \mid v_i \Rightarrow p \equiv 1 \bmod 6}}
\sum_{a \in \{1,2,3\}}
\sum_{\substack{ \widetilde{w} \in \Z_{\neq 0} \\ \widetilde{w} \equiv a \bmod 4\\ (c(a)\vert\widetilde{w}\vert)^9 \leq X/3^9(\prod_i u_i^{12} \prod_j v_j^{15})}} \mu^2 \left( 3\widetilde{w}\prod_i u_i \prod_j v_j \right),
\] and thus $\Sigma_4 = 3^{- \frac{9}{9}} \Sigma_1$. Combining these, we have
\[
\Sigma_1 + \Sigma_2 + \Sigma_3 + \Sigma_4
=
\left(1 + \frac{1}{3^{\frac{8}{3}}} + \frac{1}{81} + \frac{1}{3} \right) \Sigma_1(X)
\]
Just as in the last subsection, this sum counts each field extension multiple times and hence we must divide by the order of the automorphism group of $A = C_2 \times C_3 \times C_3$ which is 48.

Our work so far has lead us to the following result.

\begin{theorem}
The number of $C_2 \times C_3 \times C_3$-extensions $K/\Q$ with absolute discriminant $\leq X$ such that $R^1_{K/\Q} \mathbb{G}_m$ satisfies weak approximation is
\[
\frac{109+3\sqrt[3]{3}}{2^43^3\pi^2} \kappa X^{\frac{1}{9}} 
+ O_{\epsilon}\left( X^{\frac{1}{18} + \epsilon} \right).
\]
\end{theorem}

We can compare this count to the number of $A$-extensions of bounded absolute discriminant which we compute in a fairly similar manner. The explicit value for the leading constant in the following asymptotic formula is new, as far as the authors are aware.

\begin{theorem}
The number of $C_2 \times C_3 \times C_3$-extensions $K/\Q$ with absolute discriminant $\leq X$ is
\[
\frac{109+3\sqrt[3]{3}}{2^43^3\pi^2} X^{\frac{1}{9}}\prod_{p \equiv 1 \bmod 6} \left( 1 +  \frac{8}{p^{1/3}(p+1)}+  \frac{8}{p^{2/3}(p+1)} \right) 
+ O_{\epsilon}\left( X^{\frac{1}{18} + \epsilon} \right).
\]
\end{theorem}

\begin{proof}
The proof follows exactly the same lines as the previous proof but we no longer need to impose the condition that weak approximation holds in the $C_3 \times C_3$-extension. We will parametrise the extensions in the exact same way and similarly consider 4 cases depending on the ramification of the prime 3. In the first instance, when 3 does not ramify in the whole extension, we need to compute
\[
\widetilde{\Sigma_1}(X)
:=
\sum_{\substack{ \mathbf{u},\mathbf{v} \in \Z^8_{>0} \\  \prod_i u_i^{12} \prod_j v_j^{15} \leq (\log X)^{100}\\ p \mid u_i \Rightarrow p \equiv 1 \bmod 3\\ p \mid v_i \Rightarrow p \equiv 1 \bmod 6}} \mu^2\left(\prod_i u_i \prod_j v_j \right)
\sum_{a \in \{1,2,3\}}
\sum_{\substack{ w \in \Z_{\neq 0} \\ w \equiv a \bmod 4\\ \vert w \vert \leq (c(a))^{-1}(X/(\prod_i u_i^{12} \prod_j v_j^{15}))^{1/9}\\ \gcd\left( w , 3\prod_i u_i \prod_j v_j\right) = 1}} \mu^2 ( w).
\]
Thus, similarly, we have
\[
\widetilde{\Sigma_1}(X) = \frac{9}{\pi^2} X^{\frac{1}{9}}
\sum_{\substack{ \mathbf{u},\mathbf{v} \in \Z^8_{>0} \\ p \mid u_i \Rightarrow p \equiv 1 \bmod 3\\ p \mid v_i \Rightarrow p \equiv 1 \bmod 6}} \frac{\mu^2\left(\prod_i u_i \prod_j v_j \right)}{\prod_i u_i^{\frac{4}{3}} \prod_j v_j^{\frac{5}{3}}}\prod_{p \mid  \prod_i u_i \prod_j v_j} \left( 1 + \frac{1}{p} \right)^{-1} 
+
O_{\epsilon}\left( X^{\frac{1}{18} + \epsilon}\right)
\]
We can express the remaining convergent sum as an Euler product so that
\[
\widetilde{\Sigma_1}(X)
\sim
\frac{9}{\pi^2} X^{\frac{1}{9}}
\prod_{p \equiv 1 \bmod 6} \left( 1 + 8 \frac{1}{p^{1/3}(p+1)}+ 8 \frac{1}{p^{2/3}(p+1)} \right).
\]
The proof now follows the exact same lines as the previous theorem to conclude.
\end{proof}

Observe that, as we expect, the order of magnitude of these two counts are the same. Therefore, we can combine these results to establish the precise (positive) proportion of $A$-extensions whose norm one torus satisfies weak approximation.
\begin{corollary}
The proportion of $C_2 \times C_3 \times C_3$-extensions of $\Q$ whose norm one tori satisfy weak approximation is
\[
\kappa
\prod_{p \equiv 1 \bmod 6} \left( 1 +  \frac{8}{p^{1/3}(p+1)}+ \frac{8}{p^{2/3}(p+1)} \right)^{-1} .
\]
\end{corollary}

As discussed in $\S$ \ref{ss:hnp}, since we have $H^3(C_2 \times C_3\times C_3, \Z) \cong C_3$, weak approximation is satisfied if and only if the Hasse norm principle fails. Therefore, we are able to determine the precise proportion of Hasse norm principle failures in this setting as well.

\begin{corollary}
The proportion of $C_2 \times C_3 \times C_3$-extensions of $\Q$ which fail the Hasse norm principle is
\[
\kappa
\prod_{p \equiv 1 \bmod 6} \left( 1 +  \frac{8}{p^{1/3}(p+1)}+ \frac{8}{p^{2/3}(p+1)} \right)^{-1} .
\]
\end{corollary}

\end{document}